\documentclass[11pt]{article}

\usepackage{amstext,amssymb,amsmath,amsbsy}
\usepackage{hyperref}
\usepackage{amscd}
\usepackage{amsfonts}
\usepackage{indentfirst}
\usepackage{verbatim}
\usepackage{amsmath}
\usepackage{amsthm}
\usepackage{enumerate}
\usepackage{graphicx}
\usepackage{color}
\usepackage[OT1]{fontenc}
\usepackage[latin1]{inputenc}
\usepackage[english]{babel}
\usepackage{amssymb}
\usepackage{subfig}
\usepackage{algorithm}
\usepackage{algpseudocode}
\usepackage[shortlabels]{enumitem}

\newcommand{\R}{\mathbb{R}}

\newcommand{\x}{{\bf x}}

\newtheorem{Theorem}{Theorem}[section]
\newtheorem{Lemma}{Lemma}[section]

\newtheorem{Corollary}{Corollary}[section]
\newtheorem{Remark}{Remark}[section]

\newtheorem*{Assumption*}{Assumption}

\newtheorem{Problem}{Problem}[section]
\newtheorem*{Problem*}{Problem}
\numberwithin{equation}{section}

\usepackage{matlab-prettifier}

\begin{document}

\title{The time dimensional reduction method to determine the initial conditions without the knowledge of damping coefficients}

\author{
Thuy T. Le\thanks{%
Department of Mathematics, NC State University, Raleigh, NC 27695, USA, \texttt{tle9@ncsu.edu}} 
\and Linh V. Nguyen\thanks{Department of Mathematics, University of Idaho, 875 Perimeter Drive, Moscow, ID, 83844, USA \texttt{lnguyen@uidaho.edu}.} \and Loc H. Nguyen\thanks{Department of Mathematics and Statistics, University of North Carolina at
Charlotte, Charlotte, NC, 28223, USA, \texttt{loc.nguyen@charlotte.edu}.} 
\and
Hyunha Park\thanks{Department of Mathematics, George Washington University, 801 22nd St. NW, Washington, DC 20052, USA, \texttt{hpark0@gwu.edu}.}
}


\date{}
\maketitle
\begin{abstract}
	This paper aims to reconstruct the initial condition of a hyperbolic equation with an unknown damping coefficient. Our approach involves approximating the hyperbolic equation's solution by its truncated Fourier expansion in the time domain and using a polynomial-exponential basis. This truncation process facilitates the elimination of the time variable, consequently, yielding a system of quasi-linear elliptic equations. 
	To globally solve the system  without needing an accurate initial guess, we employ the Carleman contraction principle. We provide several numerical examples to illustrate the efficacy of our method. The method not only delivers precise solutions but also showcases remarkable computational efficiency.
\end{abstract}


\noindent{\it AMS subject classification: 	

}

\section{Introduction} \label{sec intr}

Let $T$ be a positive number that represents the final time and let $d \geq 1$ be the spatial dimension.
Let $u: \R^d \times (0, T) \to \R$ be the solution of
\begin{equation}
	\left\{
		\begin{array}{ll}
			u_{tt}(\x, t) + a(\x) u_t(\x, t)  = \Delta u(\x, t) &(\x, t) \in \R^d \times (0, T)			
			\\
			u(\x, 0) = f(\x) &\x \in \R^d,\\										u_t(\x, 0) = -a(\x) f(\x) &\x \in \R^d. 
		\end{array}
	\right.
	\label{main_eqn}
\end{equation}
We are interested in the following inverse problem.
\begin{Problem} \label{P:main}
	Let $\Omega$ be a bounded domain of $\R^d$ with a smooth boundary.
	Assume that $|f(\x)| > 0$ for all $\x \in \Omega$.
	 Given the measurement of lateral data
	\begin{equation}
		p(\x, t) = u(\x, t)
		\quad
		\mbox{and}
		\quad
		q(\x, t) = \partial_{\nu} u(\x, t)
		\label{data}		
	\end{equation}
	for all $(\x, t) \in \partial \Omega \times (0, T),$
	determine the function $f(\x)$ for $\x \in \Omega$.
	\label{p}
\end{Problem}

Problem \ref{p}  is an important problem arising from bio-medical imaging, called thermo/photo-acoustics tomography (see, e.g., \cite{Krugeretal:mp1995, Krugerelal:mp1999, Oraevskyelal:ps1994,XuWang:rew2006,KuchKun:ejam2008,Kuchment:siam2013}).
One sends non-ionizing laser pulses or microwave to a biological tissue under inspection (for instance, woman's breast in mamography). A part of the energy will be absorbed and converted into heat, causing a thermal expansion and a subsequence ultrasonic wave propagating in space. The ultrasonic pressures $u$ on a surface around the tissue are measured.  Although the Neumann data $\partial_\nu u$ are not directly measured in the experiment, one can find it by solving the external hyperbolic equation \cite{FPR:siam2004}. Finding the initial pressure $f$ from these measurements yields helpful structural information of the tissue.
Most of the current publications focus on standard models with non-damping and isotropic media. The methods include explicit reconstruction formulas in \cite{FPR:siam2004, FHR:siam2007,Kunyansky:inv2007,Haltmeier:cma2013, Natterer:ipi2012, Linh:ipi2009}, the time reversal method \cite{  Hristova:ip2009, HristovaKuchmentLinh:ip2006, KatsnelsonNguyen:aml2018, Stefanov:ip2009, Stefanov:ip2011}, the quasi-reversibility method \cite{ClasonKlibanov:sjsc2007, LeNguyenNguyenPowell:JOSC2021} and the iterative methods \cite{Paltaufetal:osa2002,Huangetal:IEEE2013, Paltaufetal:ip2007,Belhachmi:inv2016,HaltmeierNguyen:siam2017}.
The reader can find publications about thermo/photo-acoustics tomography for more sophisticated model involving a damping term or attenuation term \cite{Acosta:jde2018, 
Ammarietal:cm2011, Ammarielal:sp2012,  Burgholzer:pspie2007, Haltmeier:jmiv2019, HaltmeierKN:inv2017, Homan:ipi2013, Kowar:SISI2014, Kowar:sp2012, Nachman1990, NguyenKlibanov:ip2022}. 

The model under investigation \eqref{main_eqn} and Problem~\eqref{P:main} were studied in \cite{Homan:ipi2013,Palacios:inv2016,Haltmeier:jmiv2019,Palacios:ipi2022}. However, in those works, the absorption coefficient $a(x)$ was known. In this paper, in contrast, we assume that $a(x)$ is unknown. Since our focus is the inverse problem, we assume that \eqref{data} has a unique solution $u(\x, t)$. Assume further that this solution is bounded; i.e. there is an $M > 0$ such that
\begin{equation}
	|u(\x, t)| < M
	\quad 
	\mbox{for all }
	(\x, t) \in \R^d \times (0, T).
	\label{1.3}
\end{equation}
Solving Problem \ref{p} when $a$ is known is possible, see \cite{Homan:ipi2013,Palacios:inv2016,Haltmeier:jmiv2019,Palacios:ipi2022}. However, the problem becomes challenging and interesting when $a$ is not known.
\begin{enumerate}
\item Regarding ``challenging",
the challenge at hand stems from the necessity of computing two unknown functions, $a$ and $u$ while there is only one single governing equation, the hyperbolic equation in \eqref{main_eqn}. Additionally, the product $a(\x) u_t(\x, t)$ in \eqref{main_eqn} adds nonlinearity to Problem \ref{p}. Solving nonlinear problems without providing a good initial guess poses an intriguing and scientifically significant challenge for the community.
We propose to use a recently developed method, the Carleman contraction mapping method, that quickly delivers reliable solutions without requesting such a good initial guess, see \cite{Le:CONN2023, LeNguyen:jiip2022, Nguyen:AVM2023}. 
This new method is designed based on the fixed-point iteration, the contraction principle, and a suitable Carleman estimate.
	\item Regarding ``interesting", in real-world applications, the function $a(\x)$ is typically unknown as it represents the value of the damping coefficient at an internal point $\x$ in $\Omega$ where one has no access. Therefore, being able to solve Problem \ref{p} without requesting the knowledge of this internal data is a substantial contribution to the field. To the best of our knowledge, this is the first work that tackles this problem. A similar problem, which is to reconstruct the initial pressure with unknown sound speed has been studied theoretically (see \cite{HristovaKuchmentLinh:ip2006,Stefanov:ams2013,Stefanov:ipi2013,Liu:inv2015}) and numerically (e.g., \cite{Zhang:ppu2006,Matthews:siims2018}). However, theoretically sound numerical approach for this problem is still out-of-reach. 
\end{enumerate}

Due to the lack of knowledge of the coefficient $a$, Problem \ref{p} becomes nonlinear.
 Conventional approaches to computing solutions to nonlinear inverse problems typically rely on optimization techniques. However, these methods are local in nature, meaning they yield solutions only if good initial approximations of the true solutions are provided. Even in this case, local convergence is not guaranteed unless certain additional conditions are met. For a condition ensuring the local convergence of the optimization method employing Landweber iteration, we direct the reader to \cite{Hankeetal:nm1995}. 
 There is a general framework to globally solve nonlinear inverse problems, called convexification. 
The main idea of the convexification method is to include some suitable Carleman weight functions into the mismatch cost functionals, making these mismatch functionals uniformly convex.
The convexified phenomenon is rigorously proved by employing the well-known Carleman estimates.
Several versions of the convexification method  \cite{KlibanovNik:ra2017, KhoaKlibanovLoc:SIAMImaging2020, Klibanov:sjma1997, Klibanov:nw1997, Klibanov:ip2015,  KlibanovKolesov:cma2019, KlibanovLiZhang:ip2019, KlibanovLeNguyenIPI2022,  KlibanovLiZhang:SIAM2019, LeNguyen:JSC2022} have
been developed since it was first introduced in \cite{KlibanovIoussoupova:SMA1995}. 
Especially, the convexification was successfully tested with experimental data in \cite{VoKlibanovNguyen:IP2020, Khoaelal:IPSE2021, KlibanovLeNguyenIPI2022} for the inverse scattering problem in the frequency domain given only backscattering data.
We consider the convexification method as the first generation of numerical methods based on Carleman estimates to solve nonlinear inverse problems.
Although effective, the convexification method has a drawback. 
It is time-consuming. 
We, therefore, propose to apply the Carleman contraction mapping method, see \cite{Le:CONN2023, LeNguyen:jiip2022, Nguyen:AVM2023}. 
The strength of the Carleman contraction mapping method includes global and fast convergence; i.e., this method can provide reliable numerical solutions without requesting a good initial guess and the rate of the convergence is $O(\theta^n)$ where $\theta \in (0, 1)$ and $n$ is the number of iterations.
For more details about these strengths, we refer the reader to \cite{Nguyen:AVM2023}.

The Carleman contraction mapping methods developed in \cite{Le:CONN2023, LeNguyen:jiip2022, Nguyen:AVM2023} are suitable for solving nonlinear elliptic equations given Cauchy boundary data. However, the governing equation for Problem \ref{p} is hyperbolic. Hence, to Problem \ref{p} into the framework of the Carleman contraction mapping method, we must reduce the time dimension.
To achieve this, we express the function $u(\x, t)$, for $(\x, t) \in \Omega \times (0, T)$, through its Fourier coefficients $u_1(\x), u_2(\x), \dots $, where $\x \in \Omega$. These coefficients are related to the polynomial-exponential basis of $L^{2}(0, T)$ introduced in \cite{Klibanov:jiip2017}. Using straightforward algebra, we derive an approximate model consisting of a system of elliptic PDEs for these Fourier coefficients. As a result, the time dimension is reduced, and the Carleman contraction mapping method can be applied.
This process suggests our approach the name: ``the time dimensional reduction method." Another benefit of this method is its more efficient computational cost, as we are now dealing with a $d$ dimensional problem instead of a $d + 1$ dimensional one.

The paper is organized as follows. In Section \ref{sec2}, we introduced the time reduction model.
In Section \ref{sec3}, we recall a version of the Carleman contraction mapping method and its convergence.
In Section \ref{sec4}, we present several numerical results.

\section{An approximate model}\label{sec2}
In this section, we derive a system of nonlinear partial differential equations. 
The solution to this system directly yields the solution to Problem \ref{p}.
It follows from the initial conditions in \eqref{main_eqn} $u(\x, 0) = f(\x)$ and $u_t(\x, 0) = - a(\x) f(\x)$ that 
\begin{equation}
	a(\x) = -\frac{u_t(\x, 0)}{u(\x, 0)}
	\sim -\frac{u_t(\x, 0)u(\x, 0)}{|u(\x, 0)|^2 + \eta^2},
	\quad 
	\x \in \Omega
	\label{2,1}
\end{equation}
for a fixed regularization number $0 < \eta \ll 1.$

\begin{Remark}
The replacement of $-\frac{u_t(\x, 0)}{u(\x, 0)}$ by its regularized version $-\frac{u_t(\x, 0)u(\x, 0)}{|u(\x, 0)|^2 + \eta^2}$  in \eqref{2,1} is necessary. More precisely, our rate of convergence depends on $\eta$. In computation, this approximation also prevents a situation where the denominator of the fraction $-\frac{u_t(\x, 0)}{u(\x, 0)}$ becomes zero for some values of $\x$. 
\end{Remark}

Substituting \eqref{2,1} into the governing hyperbolic equation in \eqref{main_eqn}, we derive the following approximate equation
\begin{equation}
u_{tt}(\x, t) 
	 -\frac{u_t(\x, 0)u(\x, 0)}{|u(\x, 0)|^2 + \eta^2} u_t(\x, t) - \Delta u(\x, t) = 0 
	\quad
	(\x, t) \in \R^d \times (0, T).
	\label{main}
\end{equation}
To address Problem \ref{p}, we compute the solution for \eqref{main} given the lateral data of the function $u$ on $\partial \Omega \times (0, T)$. 
Once we obtain the solution $u(\x, t)$ for $(\x, t) \in \Omega \times (0, T)$ to \eqref{main}, we can set the required function $f(\x)$ as $u(\x, 0)$ for $\x \in \Omega.$ However, given that \eqref{main} is nonlocal and nonlinear, finding its solution is extremely challenging.
Currently, an efficient numerical method to handle this task is not yet developed.
We only demonstrate a numerical solver for the following approximation where the time variable and the nonlocal terms are eliminated.
The first step in removing the time dimension is to cut off the Fourier series of $u(\x, t)$ with respect to an appropriate basis of $L^2(0, T)$. 
We choose the polynomial-exponential basis $\{\Psi_n\}_{n \geq 1}$ originally introduced in \cite{Klibanov:jiip2017}.
The set $\{\Psi_n\}_{n \geq 1}$ is constructed as follows.
For any $t \in (0, T)$, we define $\phi_n(t) = t^{n - 1}e^t$. It is clear that the set $\{\phi_n\}_{n \geq 1}$ is complete in $L^2(0, T)$. 
By applying the Gram-Schmidt orthonormalization process to this set, we obtain an orthonormal basis for $L^2(0, T)$, which is denoted by $\{\Psi_n\}_{n \geq 1}$.

\begin{Remark}
	The polynomial-exponential basis set $\{\Psi_n\}_{n \geq 1}$ was initially introduced in \cite{Klibanov:jiip2017} as a tool to solve inverse problems.
We have demonstrated its effectiveness by employing to numerous inverse problems of nearly all types of equations. This includes elliptic equations \cite{VoKlibanovNguyen:IP2020,
Khoaelal:IPSE2021,
KhoaKlibanovLoc:SIAMImaging2020,
ThuyKhoaKlibanovLocBidneyAstratov:2023, NguyenVu:cm2023},
parabolic equations \cite{HaoThuyLoc:2023, Le:CONN2023, LeNguyen:jiip2022},
hyperbolic equations \cite{LeNguyenNguyenPowell:JOSC2021, NguyenNguyenTruong:camwa2022},
transport equations \cite{KlibanovNguyen:ip2019}, and full radiative transfer equation \cite{KlibanovAlexeyNguyen:SISC2019}.
In addition, we have used this basis to solve the critical task of differentiating noisy data, as described in \cite{NguyenLeNguyenKlibanov:arxiv2023}.
\end{Remark}
 
For $(\x, t) \in \Omega \times (0, T)$, we approximate
\begin{equation}
	u(\x, t) = \sum_{n = 1}^\infty u_n(\x) \Psi_n(t) \approx \sum_{n = 1}^N u_n(\x) \Psi_n(t)
	\label{2.1}
\end{equation}
for some cut-off number $N$, chosen later, where
\begin{equation}
	u_n(\x) = \int_{0}^T u(\x, t) \Psi_n(t)dt, \quad n \geq 1.	
	\label{2.3}	
\end{equation}
Substituting \eqref{2.1} into the governing equation into \eqref{main} gives
\begin{equation}
	\sum_{n = 1}^N u_n(\x) \Psi_n''(t) 
	-\frac{\big[\sum_{l = 1}^N u_l(\x) \Psi_l'(0)\big]\big[\sum_{l = 1}^N u_l(\x) \Psi_l(0)\big]}{\big|\sum_{l = 1}^N u_l(\x) \Psi_l(0)\big|^2 +\eta^2} \sum_{n = 1}^N u_n(\x) \Psi_n'(t) 
	- \sum_{n = 1}^N \Delta  u_n(\x) \Psi_n(t)  = 0
	\label{2.4}
\end{equation}
for $(\x, t) \in \Omega \times (0, T).$
For each $m \in \{1, \dots, N\},$ multiply $\Psi_m(t)$ to both sides of \eqref{2.4} and then integrate the resulting equation. We obtain
\begin{multline}
	 \sum_{n = 1}^N u_n(\x) \int_0^T \Psi_n''(t) \Psi_m(t) dt 
	-\frac{\big[\sum_{l = 1}^N u_l(\x) \Psi_l'(0)\big]\big[\sum_{l = 1}^N u_l(\x) \Psi_l(0)\big]}{\big|\sum_{l = 1}^N u_l(\x) \Psi_l(0)\big|^2 +\eta^2} \sum_{n = 1}^N u_n(\x) \int_0^T \Psi_n'(t) \Psi_m(t) dt
	\\
	- \sum_{n = 1}^N \Delta  u_n(\x)\int_0^T \Psi_n(t)\Psi_m(t) dt  = 0
	\label{2.5}
\end{multline}
for $\x \in \Omega.$
Define the vector $U = (u_1, \dots, u_N)^{\rm T}$. Since $\{\Psi_n\}_{n \geq 1}$ is an orthonormal basis of $L^2(0, T)$, we can deduce from \eqref{2.5} that
\begin{equation}
	\Delta U(\x) -  SU(\x)= F(\x, U(\x))
	\quad
	\mbox{for all } \x \in \Omega
	\label{2.6}
\end{equation}
where the matrix $S$ is given by
\begin{equation}
	S = (s_{mn})_{m, n = 1}^N =  \Big(\int_0^T \Psi_n''(t) \Psi_m(t)dt\Big)_{m, n = 1}^N
\end{equation}
and the function $F: \Omega \times \R^N \to \R^N$ is defined as 
\begin{equation}
	F(\x, U(\x)) = 
	-\frac{\big[\sum_{l = 1}^N u_l(\x) \Psi_l'(0)\big]\big[\sum_{l = 1}^N u_l(\x) \Psi_l(0)\big]}{\big|\sum_{l = 1}^N u_l(\x) \Psi_l(0)\big|^2 +\eta^2} \sum_{n = 1}^N u_n(\x) \int_0^T \Psi_n'(t) \Psi_m(t) dt.
	\label{2,9}
\end{equation}
Due to \eqref{2.1} and the boundedness of $u(\x, t)$, see \eqref{1.3}, the vector $U$ is bounded, say, there is a positive number ${\bf M}$ depending only on $M,$ $N$,  $\{\Psi_n\}_{n = 1}^N$ and $T$ such that
\begin{equation}
	|U(\x)| \leq {\bf M}
	\quad
	\mbox{for all } \x \in \Omega.
	\label{boundednessU}
\end{equation}

On the other hand, it follows from \eqref{data} and \eqref{2.3} that
\begin{equation}
	U(\x) = \Big(\int_0^T p(\x, t) \Psi_m(t)dt\Big)_{m = 1}^N,
	\quad
	\mbox{and }
	\quad
	\partial_{\nu} U(\x) = \Big(\int_0^T q(\x, t) \Psi_m(t)dt\Big)_{m = 1}^N
\end{equation}
for all $\x \in \partial \Omega$.
So, we have derived the following time-reduction model 
\begin{equation}
	\left\{
		\begin{array}{ll}
			\Delta U(\x) -  S U(\x) = F(\x, U(\x)) &\x \in \Omega \\
			U(\x) = \Big(\int_0^T p(\x, t) \Psi_m(t)dt\Big)_{m = 1}^N &\x \in \partial \Omega,\\
			\partial_{\nu} U(\x) = \Big(\int_0^T q(\x, t) \Psi_m(t)dt\Big)_{m = 1}^N &\x \in \partial \Omega.
		\end{array}
	\right.
	\label{2.10}
\end{equation}

\begin{Remark}[The validity of the approximation model \eqref{2.10}]
	Due to the truncation in \eqref{2.1}, and the term-by-term differentiation to obtain \eqref{2.4}, problem \eqref{2.10} is not precise. 	
	Proving the convergence of this model as $N \to \infty$ is extremely challenging.
	Since this paper focuses on computation, we do not address this issue here. 
	Instead, we assume that \eqref{2.10} well-approximates the model for the Fourier coefficients $U(\x) = (
	\begin{array}{rrr}
		u_1(\x) & \dots& u_N(\x)
	\end{array})^{\rm T}$ of the function $u(\x, t)$.
	Although the validity of this approximation model is not theoretically proven, we numerically observe its strength.
	In fact,
	similar approximations were successfully applied in \cite{VoKlibanovNguyen:IP2020, Khoaelal:IPSE2021, KhoaKlibanovLoc:SIAMImaging2020, ThuyKhoaKlibanovLocBidneyAstratov:2023} in which we solved the highly nonlinear and severely ill-posed inverse scattering problem with backscattering data experimentally measured by microwave facilities built at the University of North Carolina at Charlotte.
	The successful applications of several versions of this approximation with highly noisy simulated data can be found at \cite{KlibanovNguyen:ip2019,  LeNguyenNguyenPowell:JOSC2021, NguyenNguyenTruong:camwa2022, NguyenVu:cm2023, KlibanovAlexeyNguyen:SISC2019}.
\end{Remark}

\begin{Remark}[The choice of the basis $\{\Psi_n\}_{n \geq 1}$]
The use of the basis $\{\Psi_n\}_{n \geq 1}$ is crucial to the efficacy of our method.  One may question why we have chosen this specific basis out of countless alternatives for the Fourier expansion in \eqref{2.1}. The answer lies in the limitations of more common bases like Legendre polynomials or trigonometric functions. These bases typically commence with a constant function, whose derivatives are identically zero. As a result, the corresponding Fourier coefficient $u_{1}(\x)$ in the sums $ \sum_{n = 1}^N u_{n}(x) \Psi_n'(t)$ in \eqref{2.3} is overlooked, causing some inaccuracy of the outcome. The basis $\{\Psi_n\}_{n \geq 1}$ is suitable for \eqref{2.3} as it fulfills the necessary condition wherein the derivatives of $\Psi_n$, for $n \geq 1$, are not constantly zero.
The effectiveness of this basis is well-documented in various research, such as in \cite{LeNguyenNguyenPowell:JOSC2021, NguyenLeNguyenKlibanov:arxiv2023}. 
In the study \cite{LeNguyenNguyenPowell:JOSC2021}, we use the basis $\{\Psi_n\}_{n \geq 1}$ and a traditional trigonometric basis to expand wave fields and address problems in photo-acoustic and thermo-acoustic tomography. Our results indicated a notably superior performance from the polynomial-exponential basis. In \cite{NguyenLeNguyenKlibanov:arxiv2023}, we employed Fourier expansion to compute the derivatives of data affected by noise. Our findings confirmed that the basis $\{\Psi_n\}_{n \geq 1}$ outperformed the trigonometric basis in terms of accuracy when differentiating term-by-term of the Fourier expansion in solving ill-posed problems.
\end{Remark}

\begin{Remark}
The task of solving Problem \ref{p} is reduced to the problem of computing solution to \eqref{2.10}.
In fact, let $U^{\rm comp} = (u_1^{\rm comp}, \dots, u_N^{\rm comp})^{\rm T}$ denote the computed solution to \eqref{2.10}. Then, since $f(\x) = u(\x, 0)$ and due to \eqref{2.1}, we set the desired solution to Problem \ref{p} as
\begin{equation}
	f^{\rm comp}(\x) = \sum_{n = 1}^N u^{\rm comp}_n(\x) \Psi_n(0)
	\quad
	\mbox{for } \x \in \Omega.
	\label{2.13}
\end{equation}
\label{rem24}
\end{Remark}

We recall the Carleman contraction method to solve \eqref{2.10} in the next section. Some versions of this method can be found in \cite{Le:CONN2023, LeNguyen:jiip2022, Nguyen:AVM2023}.

\section{The Carleman contraction method}\label{sec3}

Some versions of the Carleman contraction method, which were established in \cite{Le:CONN2023, LeNguyen:jiip2022, Nguyen:AVM2023}, rely on Carleman estimates. For the reader's ease of reference, we briefly recall the version in \cite{Nguyen:AVM2023} here we will apply it to solve \eqref{2.10}.
%

\begin{Lemma}[Carleman estimate]
Fix a point $\x_0 \in \R^d \setminus \Omega$. Define $r(\x) = |\x - \x_0|$ for all $\x \in \Omega.$
	Let $b > \max_{\x \in \overline \Omega} r(\x)$ be a fixed constant.
There exist  positive constants $\beta$ depending only on $\x_0$, $\Omega$, $\Lambda,$  and $d$ such that for all function $v \in C^2(\overline \Omega)$ satisfying
	\begin{equation}
		v(\x) = \partial_{\nu} v(\x) = 0 \quad \mbox{for all } 
		\x \in \partial \Omega,
		\label{3.1}
	\end{equation}
	the following estimate holds true
	\begin{equation}
		\int_{\Omega} e^{2\lambda r^{-\beta}(\x)}\vert{\rm Div}(A \nabla v) \vert^2d\x
		\geq
		 C\lambda  \int_{\Omega}  e^{2\lambda r^{-\beta}(\x)}\vert\nabla v(\x)\vert^2\,d\x
		+ C\lambda^3  \int_{\Omega}   e^{2\lambda r^{-\beta}(\x)}\vert v(\x)\vert^2\,d\x	
		\label{Car est}	
	\end{equation}
	for all $\lambda \geq \lambda_0$. 
	Here, $\lambda_0 = \lambda_0( \x_0,  \Omega, A, d, \beta)$ and $C = C( \x_0,  \Omega, A, d, \beta) > 0$ depend only on the listed parameters.
	\label{lemma carl}
\end{Lemma}

Lemma \ref{lemma carl} can be straightforwardly derived from \cite[ Lemma 5]{MinhLoc:tams2015}. For a detailed exposition of the proof, we direct the reader to \cite[Lemma 2.1]{LeNguyenTran:CAMWA2022}.
Another approach to derive \eqref{Car est}, using a different Carleman weight function, involves the application of the Carleman estimate from \cite[Chapter 4, Section 1, Lemma 3]{Lavrentiev:AMS1986} for generic parabolic operators.
The methodology to derive \eqref{Car est} via \cite[Chapter 4, Section 1, Lemma 3]{Lavrentiev:AMS1986} resembles the one in \cite[Section 3]{LeNguyenNguyenPowell:JOSC2021}, albeit with the Laplacian swapped out for the operator ${\rm Div} (A\nabla \cdot)$.
We would like to specifically highlight to the reader the varied forms of Carleman estimates for all three types of differential operators (elliptic, parabolic, and hyperbolic) and their respective applications in inverse problems and computational mathematics \cite{BeilinaKlibanovBook, BukhgeimKlibanov:smd1981, KlibanovLiBook, LocNguyen:ip2019}.
Additionally, it is noteworthy that certain Carleman estimates remain valid for all functions $v$ that satisfy $v\vert_{\partial \Omega} = 0$ and $\partial_{\nu} v\vert_{\Gamma} = 0$, where $\Gamma$ constitutes a portion of $\partial \Omega$. Examples can be seen in \cite{KlibanovNguyenTran:JCP2022, NguyenLiKlibanov:2019}. These Carleman estimates are applicable to the resolution of quasilinear elliptic PDEs given the data on only a part of $\partial \Omega$.

We will seek the solution $U$ to \eqref{2.10} in 
the set of admissible solutions
\begin{multline}
	H = \Big\{h \in H^s(\Omega)^N: 
		\|h\|_{L^{\infty}(\Omega)} \leq {\bf M},  
		h|_{\partial \Omega} = \Big(\int_0^T p(\x, t) \Psi_m(t)dt\Big)_{m = 1}^N
		\\
	\mbox{ and }		
		  \partial_{\nu} h|_{\partial \Omega} = \Big(\int_0^T q(\x, t) \Psi_m(t)dt\Big)_{m = 1}^N	
	\Big\}
	\label{H}
\end{multline}
where $s$ is an integer with $s > \lceil d/2 \rceil + 2$ and ${\bf M}$ is the number in \eqref{boundednessU}.
Throughout the paper, we assume that $H$ is nonempty.
Let $\beta$ and $\lambda_0$ be the numbers in Lemma \ref{lemma carl}.
Recall the set $H$ as in \eqref{H}.
For $\lambda > \lambda_0,$
we define the map $\Phi_{\lambda}: H \to H$
\begin{equation}
	\Phi_{\lambda, \epsilon}(V) = \underset{\varphi \in H}{\min} 
	\int_{\Omega} e^{2\lambda  r^{-\beta}(\x)}\big|\Delta \varphi - S\varphi - F(\x, V(\x))\big|^2d\x
	+ \epsilon \|\varphi\|_{H^s(\Omega)^N}^2
	\label{3.4}
\end{equation}
where $\epsilon > 0$ is a regularization parameter.
The map $\Phi_\lambda$ is well-defined.
In fact,
for each $V \in H$, the functional
\[
	J_{\lambda, \epsilon}(V): H \to \R, \quad \varphi \mapsto \int_{\Omega} e^{2\lambda  r^{-\beta}(\x)}\big|\Delta \varphi - S\varphi - F(\x, V(\x))\big|^2d\x
	+ \epsilon \|\varphi\|_{H^s(\Omega)^N}^2
\]
is strictly convex. It has a unique minimizer on a closed and convex set $H$ of $H^s(\Omega)^N$.
We refer the reader to \cite[Remark 3.1]{Nguyen:AVM2023} for more details. A similar argument for the well-posedness of the map $\Phi_\lambda$ can be found in \cite[Theorem 4.1]{Le:CONN2023}.

\begin{Remark}[The Carleman quasi-reversibility method]
   Consider a vector-valued function $V$ in $H$. Define $\varphi$ as $\Phi_{\lambda, \epsilon}(V)$.
As $\varphi \in H$ minimizes the function $J_{\lambda, \epsilon}(V)$, it can be informally said that computing $\varphi$ is about finding a solution for the following system of equations:
    \begin{equation}
    \left\{
        \begin{array}{ll}
            e^{\lambda r^{-\beta}(\x)}[\Delta  \varphi  - S\varphi - F(\x, V(\x))] = 0 &\x \in \Omega,  \\
             \varphi(\x) = \Big(\int_0^T p(\x, t) \Psi_m(t)dt\Big)_{m = 1}^N &\x \in \partial \Omega,\\
            \partial_{\nu} \varphi(\x)  =\Big(\int_0^T q(\x, t) \Psi_m(t)dt\Big)_{m = 1}^N &\x \in \partial \Omega.
        \end{array}
    \right.
    \label{main_eqn1}
\end{equation}
Due to the presence of the regularization term $\epsilon \|\varphi\|_{H^s(\Omega)}^2$, we refer to $\varphi$ as the regularized solution of problem \eqref{main_eqn1}.
The technique for determining the regularized solution to the linear equation \eqref{main_eqn1} by minimizing $J_{\lambda, \epsilon}(V)$ is named the Carleman quasi-reversibility method. The name of this method is suggested by the existence of the Carleman weight function in the formulation of $J_{\lambda, \epsilon}(V)$, as well as the use of the quasi-reversibility technique to address linear PDEs with Cauchy data. We refer the reader to \cite{LattesLions:e1969} for original work regarding the quasi-reversibility method.
    \label{rem_quasi}
\end{Remark}

For $\epsilon > 0$ and $\lambda > \lambda_0$, define the norm
\begin{equation}
    \|\varphi\|_{\lambda, \epsilon} = 
    \Big(\int_{\Omega} e^{2\lambda r^{-\beta}(\x)}(|\varphi|^2 + |\nabla \varphi|^2) d\x
    \Big)^{1/2} + \frac{\epsilon}{\lambda} \|\varphi\|_{H^s(\Omega)^N}
    \label{norm}
\end{equation}
for all $\varphi \in H^s(\Omega)^N.$
The following result holds. 
\begin{Theorem}
	Let $\x_0$, $\beta$, and $\lambda_0$ be as in Lemma \ref{lemma carl}. Then,
	there is a number $C$ depending only on  $\x_0,$ $\Omega,$ $\beta$, $T$, $\{\Psi_n\}_{n = 1}^N$, ${\bf M}$, $\eta$ and $d$ such that for all $\lambda > \lambda_0$
    \begin{equation}
        \Vert\Phi(u) - \Phi(v)\Vert_{\epsilon, \lambda} \leq \sqrt{\frac{C}{\lambda}} \Vert u - v\Vert_{\epsilon, \lambda}
        \label{contraction}
    \end{equation}
    for all $u, v \in H^s(\Omega)^N.$
    \label{thm31}
\end{Theorem}

\begin{proof}
	Since the function $F$ is smooth, it is Lipschitz in the bounded domain $H$. In other words, we can find a constant $C$ depending on $F$ and ${\bf M}$, or more precisely, on the parameters listed in the statement of Theorem \ref{thm31}
such that
\[
	|F(\x, V_1(\x)) - F(\x, V_2(\x))| \leq C|V_1(\x) - V_2(\x)|
	\quad 
	\mbox{for all }
	\x \in \overline \Omega.
\]
	We now apply Theorem 3.1 in \cite{Nguyen:AVM2023} to obtain \eqref{contraction}.
\end{proof}

\begin{Corollary}
    Choose $\lambda \gg 1$ such that $\theta = \sqrt{\frac{C}{\lambda}} \in (0, 1)$. It follows from  \eqref{contraction} $\Phi_{\lambda, \epsilon}$ is a contraction map with respect to the norm $\|\cdot\|_{\lambda, \epsilon}$.
    \label{cor1}
\end{Corollary}

Due to Corollary \ref{cor1}, $\Phi_{\lambda, \epsilon}$, $\lambda \gg 1$ and $\epsilon > 0$, has a unique fixed-point in $H$, named as $\overline U$. The vector $\overline U$ can be computed as the limit of the sequence $\{U_n\}_{n \geq 0}$ as $n \to \infty$. The the sequence $\{U_n\}_{n \geq 0}$ is defined by
\begin{equation}
	\left\{
		\begin{array}{l}
			U_0 \mbox{ is an arbitrary vector-valued function in } H,\\
			U_{n + 1} = \Phi_{\lambda, \epsilon}(U_n), \quad n \geq 0.
		\end{array}
	\right.
	\label{sequence}
\end{equation}

We next discuss how close the limit $\overline U$ to the true solution to \eqref{2.10}.
Assume that the analog of \eqref{2.10}, which is read as
\begin{equation}
	\left\{
		\begin{array}{ll}
			\Delta U(\x) -  S U(\x) = F(\x, U(\x)) &\x \in \Omega \\
			U(\x) = \Big(\int_0^T p^*(\x, t) \Psi_m(t)dt\Big)_{m = 1}^N &\x \in \partial \Omega,\\
			\partial_{\nu} U(\x) = \Big(\int_0^T q^*(\x, t) \Psi_m(t)dt\Big)_{m = 1}^N &\x \in \partial \Omega.
		\end{array}
	\right.
\end{equation}
has a unique solution in $H$ where $p^*$ and $q^*$ are the noiseless versions of the measurements $p$ and $q$ respectively.
The corresponding noisy data with noise level $\delta > 0$ are denoted by $p^{\delta}$ and $q^{\delta}$.
Define the set 
\[
	E^\delta = \{{\bf e}_{H^s(\Omega)^N}: {\bf e}|_{\partial \Omega} = p^\delta - p^*
	\mbox{  and } \partial_{\nu} {\bf e}|_{\partial \Omega} = q^{\delta} - q^*\}
\] 
By noise level, we mean that $E^\delta \not = \emptyset$ and 
\begin{equation}
	\inf \{\|{\bf e}\|_{H^s(\Omega)^N}: {\bf e} \in E\} < \delta.
	\label{noise}
\end{equation}
By \eqref{noise}, there is a vector function ${\bf e} \in H^s(\Omega)^N$ such that
\begin{equation}
	\left\{
		\begin{array}{l}
			{\bf e}_{H^s(\Omega)^N} < 2\delta 
			\\			
			{\bf e}|_{\partial \Omega} = p^\delta - p^*
			\\
			\partial_{\nu} {\bf e}|_{\partial \Omega} = q^{\delta} - q^*
		\end{array}
	\right.
	\label{ee}
\end{equation}

\begin{Remark}
    The condition that $E^{\delta} \not = \emptyset$ and \eqref{noise}, as well as, \eqref{ee} imply that the differences $p^{\delta} - p^*$ and $q^{\delta} - q^*$ are traces of smooth vector-valued functions on $\partial \Omega$. This implies the noise must exhibit smooth characteristics, which may not always be true in real-world scenarios.
The demand for smoothness is a crucial component in the convergence result in Theorem \ref{thm2}. In real-world applications, the data can be smoothed using several established techniques, such as spline curves or the Tikhonov regularization approach.
Nevertheless, this smoothing step can be relaxed during numerical investigations. This implies that our method's practical application may surpass its theoretical proof. In our numerical tests, we do not have to smooth out the noisy data. Instead, we directly derive the desired numerical solutions to \eqref{main_eqn} using the noisy (raw) data of the form
    \begin{equation}
        p^{\delta} = p^*(1 + \delta {\rm rand})
        \quad 
        \mbox{and }
        \quad
        q^{\delta} = q^*(1 + \delta {\rm rand})
        \label{random}
    \end{equation}
    where ${\rm rand}$ is a function taking uniformly distributed random numbers in the range $[-1, 1].$
    \label{rem41}
\end{Remark}

We have the theorem.
\begin{Theorem}
     Recall $\beta$ and $\lambda_0$ as in Lemma \ref{lemma carl}.
    Let $\lambda \geq \lambda_0$  be such that \eqref{Car est} holds true and the number 
     $\theta$ in Corollary \ref{cor1} is in $(0, 1)$. 
    Let $\{U_n\}_{n \geq 0} \subset H$ be the sequence defined in \eqref{sequence} and $H$ is defined in \eqref{H} with $p$ and $q$ being replaced by $p^\delta$ and $q^\delta$ respectively. Then,
    \begin{equation}
            \|\overline U - U^*\|_{\epsilon, \beta, \lambda}^2 \leq
            \frac{C}{\lambda} \Big[
            \int_{\Omega} e^{2\lambda r^{-\beta}(\x)} \Big[ |\Delta {\bf e}(\x))|^2
		 +
	        |{\bf e}(\x)|^2 
	        + |\nabla {\bf e}(\x)|^2 
	    \Big] d\x
	    +  \epsilon \|{\bf e}\|_{H^p(\Omega)}^2 +  \epsilon \|u^*\|_{H^p(\Omega)}^2
            \Big]
            \label{u_star_conve}
        \end{equation}
    where $C$ is a positive constant depending only on $\x_0,$ $\Omega,$ $\beta$, $T$, $\{\Psi_n\}_{n = 1}^N$, ${\bf M}$, $\eta$ and $d$.
    \label{thm2}
\end{Theorem}

Theorem \ref{thm2} has a broader scope than the theorem presented in \cite[Theorem 4.1]{Nguyen:AVM2023} in the sense that Theorem \ref{thm2} can be applied to solve systems, whereas \cite[Theorem 4.1]{Nguyen:AVM2023} pertains to a single equation. Nonetheless, the proof of \cite[Theorem 4.1]{Nguyen:AVM2023} can be readily adapted to substantiate Theorem \ref{thm2}.

Estimate \eqref{u_star_conve} is interesting in the sense that it, together with \eqref{ee}, guarantees that $\overline u$ tends to $u^*$ as the noise level $\delta$ and the regularization parameter $\epsilon$ tends to $0$.
In particular,
fix $\lambda \gg 1$ and set $\epsilon = O(\delta^2)$, the convergence rate is Lipschitz. 

\section{Numerical study}\label{sec4}

It is suggested by Remark \ref{rem24}, Theorem \ref{thm31}, and Theorem \ref{thm2} a Carlaman contraction method to solve Problem \ref{p}.
We present this method in
Algorithm \ref{alg}.

\begin{algorithm}[h!]
\caption{\label{alg}The procedure to compute the numerical solution to \eqref{main_eqn}}
	\begin{algorithmic}[1]
	\State \label{s_chooseN} Choose a cut-off number $N$.
	\State \label{s1} Choose Carleman parameters $\x_0,$ $\beta$, and $\lambda$ and a regularization parameter $\epsilon$.
	 \State \label{s2} Set $n = 0$. Choose an arbitrary initial solution $U_0 \in H.$
		\State \label{step update} 		
		Compute $U_{n + 1} = \Phi_{\lambda, \epsilon}(U_n)$ where $ \Phi_{\lambda, \epsilon}$ is defined in \eqref{3.4}.
	\If {$\|U_{n + 1} - U_n\|_{L^2(\Omega)} > \kappa_0$ (for some fixed number $\kappa_0 > 0$)}
		\State Replace $n$ by $n + 1.$
		\State Go back to Step \ref{step update}.
	\Else	
		\State Set the computed solution $U^{\rm comp} = U_{n + 1}.$
	\EndIf
	\State \label{s11} Write $U^{\rm comp} = (u_1^{\rm comp}, \dots, u_N^{\rm comp})^{\rm T}$ and set the desired solution as in \eqref{2.13}.

\end{algorithmic}
\end{algorithm}


In this section, we display several numerical examples in 2D computed by Algorithm \ref{alg}.
In all tests below, we set $\Omega = (-1, 1)^2$ and $T = 1$.

\subsection{Discretization and data simulation}

To generate noisy data $p$ and $q$ for Problem \ref{p}, we need to solve the hyperbolic equation \eqref{main_eqn}.
However,  computing the solution to \eqref{main_eqn} on the whole domain $\R^2 \times (0, T)$ is complicated. 
For simplicity, we replace $\R^d$ with the domain $G = (-3, 3)^2$ that contains the computational domain $\Omega$.
On $G$, we arrange a uniform grid of points 
\[
	\mathcal{G} = \{
		\x_{ij} = (x_i = -1 + (i - 1) d_\x, y_j = -1 + (j - 1) d_\x):
		1 \leq i, j \leq N_\x 
	\}
\]
where $N_\x = 241$ and $d_\x = \frac{2}{N_\x - 1} = .025$.
We discretize the time domain $(0, T)$ by the partition
\[
	\mathcal{T} = \{t_l = (l-1)d_t: 1 \leq l \leq N_t\}
\]
where $N_t = 201$ and $d_t = T/(N_t - 1) = 0.005$.

We write the governing partial differential equation in the finite difference scheme as
\begin{equation}
	\frac{u(\x_{ij}, t_{l+1}) - 2u(\x_{ij}, t_{l}) + u(\x_{ij}, t_{l-1})}{d_t^2}
	+ a(\x_{ij}) \frac{u(\x_{ij}, t_{l+1}) - u(\x_{ij}, t_{l})}{d_t}
	= \Delta^{d_\x} u(\x_{ij}, t_l),
	\label{4.1}
\end{equation}
for all $\x_{ij} \in \mathcal{G}$ and $t_l \in \mathcal{T}$
where
\[
	\Delta^{d_\x} u(\x_{ij}, t_l) = \frac{u(\x_{(i+1)j}, t_{l}) + u(\x_{(i-1)j}, t_l) + u(\x_{i(j - 1)}, t_{l}) + u(\x_{i(j + 1)}, t_{l}) - 4u(\x_{ij}, t_{l}) }{d_\x^2}.
\]
Solving \eqref{4.1} for $u(\x_{ij}, t_{l+1})$ gives
\begin{equation}
	u(\x_{ij}, t_{l+1}) =  
	\frac{\frac{1}{d_t^2} (2 u(\x_{ij}, t_{l}) - u(\x_{ij}, t_{l - 1}))
	+ \frac{a(\x_{ij})}{d_t} u(\x_{ij},t_{l})
	+ \Delta^{d_\x} u(\x_{ij}, t_l)}{\frac{1}{dt^2} + \frac{a(\x_{ij})}{d_t}}
	\label{4.2}
\end{equation}
for all $\x_{ij} \in \mathcal{G}$ and $t_l \in \mathcal{T}$.
Due to the initial conditions in \eqref{main_eqn}, we can compute
\[
	u(\x_{ij}, t_1) = f(\x_{ij}),
	\quad
	u(\x_{ij}, t_2) = f(\x_{ij}) - a(\x_{ij}) f(\x_{ij}) d_t
\]
for all $\x_{ij} \in \mathcal{G}.$
Using \eqref{4.2}, we can compute $u(\x_{ij}, t_3), u(\x_{ij}, t_4), \dots$, for all $\x_{ij} \in \mathcal{G}.$
This method of solving hyperbolic equations is well-known as the explicit method.
Having the function $u$ on $\mathcal G \times \mathcal{T}$ in hand, we can extract the noiseless data $p^*$ and $q^*$ on $(\mathcal G \cap \partial \Omega) \times \mathcal{T}$ easily.
The corresponding noisy data  $p^\delta$ and $q^\delta$  are computed as in \eqref{random}.
In our computation, we take $\delta = 10\%.$

\subsection{Implementation}

We now present some remarkable points in computing solutions to the inverse problem. The first step is to find a suitable number $N$ for \eqref{2.1}.


 We employ the strategy in \cite{HaoThuyLoc:2023} to determine the cut-off number $N$ in step \ref{s_chooseN}. 
More precisely, the knowledge of given data $p(\x, t)$, $(\x, t) \in \partial \Omega \times (0, T)$ is helpful in determining the cut-off numbers for step \ref{s_chooseN} of Algorithm \ref{alg}. The procedure is based on a trial-and-error process. 
Inspired by \eqref{2.1} and the fact that $u(\x, t) = p(\x, t)$ for all $(\x, t) \in \partial  \Omega \times (0, T)$, for each $N \geq 1$, we define the function $\varphi: \mathbb{N} \to \R$  that represents the relative distance between the data $p$ and the truncation of its Fourier expansion in \eqref{2.1}.
The function $\eta$ is defined as follows
\begin{equation}
	\eta(N)= \frac{\|p(\x)- \sum_{n = 1}^N p_{n}(\x)\Psi_n(t)\|_{L^{\infty}(\partial \Omega \times (0, T))}}{\|p(\x)\|_{L^{\infty}(\partial \Omega \times (0, T))}}
	\quad
	\mbox{where}
	\quad
	p_{n}(\x) = \int_{0}^T p(\x, t )\Psi_n(t) dt
	\label{4.3}
\end{equation}
 We increase the numbers $N$ until $\eta(N)$ is sufficiently small.
 In our computation, $N = 35.$


The selection of artificial parameters in Algorithm \ref{alg} is facilitated by a process of trial and error. We use a reference test (Test 1 below) where the correct solution is already known. Manually, we select $\x_0,$ $\lambda$, $\lambda,$ $\epsilon$, $\kappa_0$ such that the solution computed through Algorithm \ref{alg} is satisfactory. Then, we take these parameters for all other tests where the true solutions are not known.
In our computation, $\x_0 = (0, 5.5),$ $\beta = 25$, and $\lambda = 45$.
The regularization parameter $\epsilon$ is $10^{-5}.$ 
We also set the number $\eta$ in \eqref{2,1} is $10^{-11}$.

Step \ref{s2} of Algorithm \ref{alg} requires us to choose a vector-valued function $U_0$ in $H$. 
A straightforward approach for computing such a function involves solving the linear problem, which results from excluding the nonlinearity $F$ from \eqref{2.10},
namely,
\begin{equation}
    \left\{
        \begin{array}{ll}
             \Delta U_{0}(\x) - SU_0(\x)  = 0 &\x \in \Omega,  \\
             U_{0}(\x) = \Big(\int_0^T p(\x, t) \Psi_m(t)dt\Big)_{m = 1}^N &\x \in \partial \Omega,\\
             \partial_{\nu} U_{0}(\x)  = \Big(\int_0^T q(\x, t) \Psi_m(t)dt\Big)_{m = 1}^N &\x \in \partial \Omega,
        \end{array}
    \right.
    \label{main_eqn41}
\end{equation}
 using the Carleman quasi-reversibility method, as indicated in Remark \ref{rem_quasi}.
Similarly, in Step \ref{step update}, we aim to minimize $J_{\lambda, \epsilon}(U_n)$ in $H$, $n \geq 0$. As in Remark \ref{rem_quasi}, the minimizer we obtain, $U_{n+1}$, can be viewed as the regularized solution to the following system
\begin{equation}
    \left\{
        \begin{array}{ll}
             \Delta U_{n + 1}(\x) - SU_{n + 1} + F(\x, U_n(\x)) = 0 &\x \in \Omega,  \\
             U_{n + 1}(\x) = \Big(\int_0^T p(\x, t) \Psi_m(t)dt\Big)_{m = 1}^N &\x \in \partial \Omega,\\
             \partial_{\nu} U_{n+1}(\x)  = \Big(\int_0^T q(\x, t) \Psi_m(t)dt\Big)_{m = 1}^N &\x \in \partial \Omega.
        \end{array}
    \right.
    \label{main_eqn40}
\end{equation}

The details in implementation to compute the regularized solution $U_0$ and $U_{n + 1}$, $n \geq 0$ by solving \eqref{main_eqn41} and \eqref{main_eqn40} respectively, were presented in \cite{LeNguyen:jiip2022, Nguyen:CAMWA2020, Nguyens:jiip2020}, in which we use MATLAB in-built linear least squares optimization package. We will not provide these details again in this document.

\subsection{Numerical examples}

{\bf Test 1.} We test the case when the ``donut-shaped" function $f^{\rm true}$ is given by
\[
	f^{\rm true}(x, y) = \left\{
	\begin{array}{ll}
		2 &\mbox{if } 0.15^2 < (x - 0.35)^2 + y^2 < 0.6^2\\
		1 &\mbox{otherwise}
	\end{array}
	\right.
	\mbox{for all } (x, y) \in \Omega.
\]

The true and computed source functions $f$ are displayed in Figure \ref{fig1}.

\begin{figure}[h!]
	\subfloat[\label{fig1a}]{\includegraphics[width=.3\textwidth]{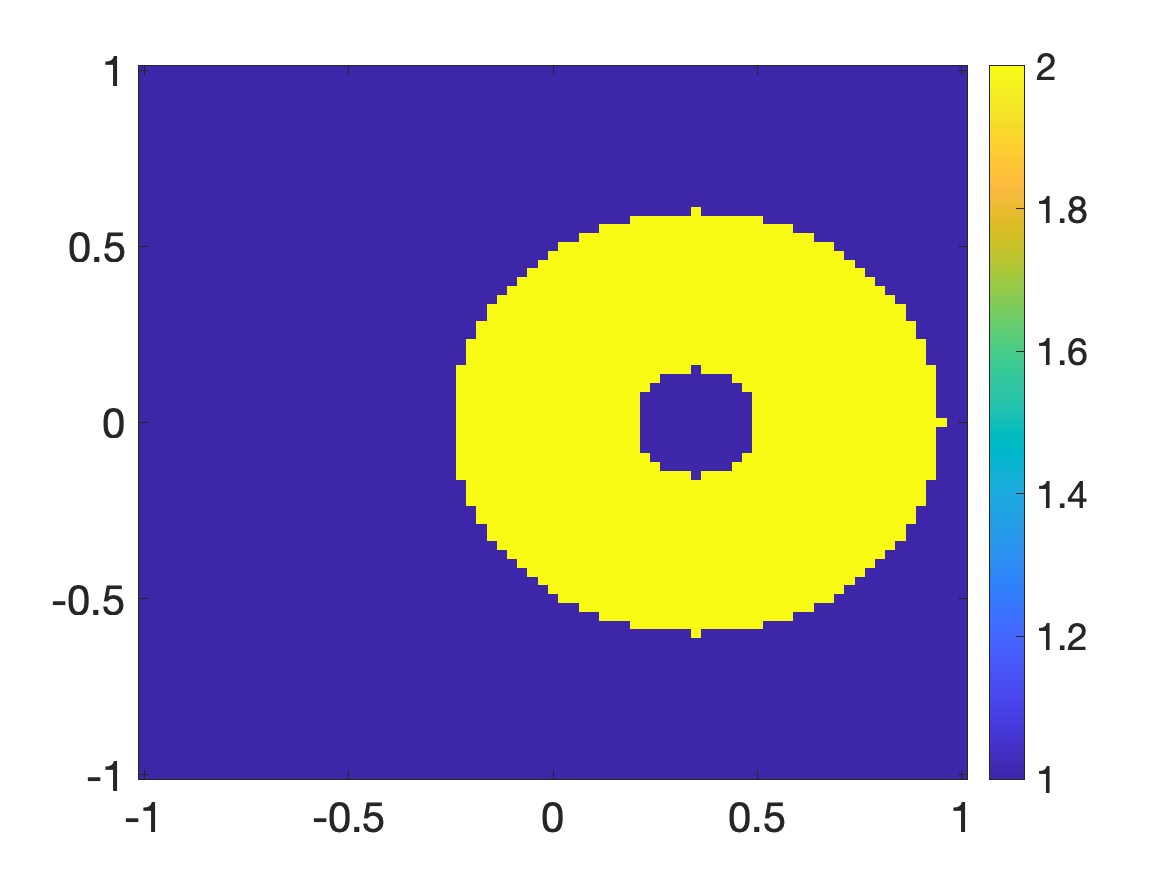}}
	\quad
	\subfloat[\label{fig1b}]{\includegraphics[width=.3\textwidth]{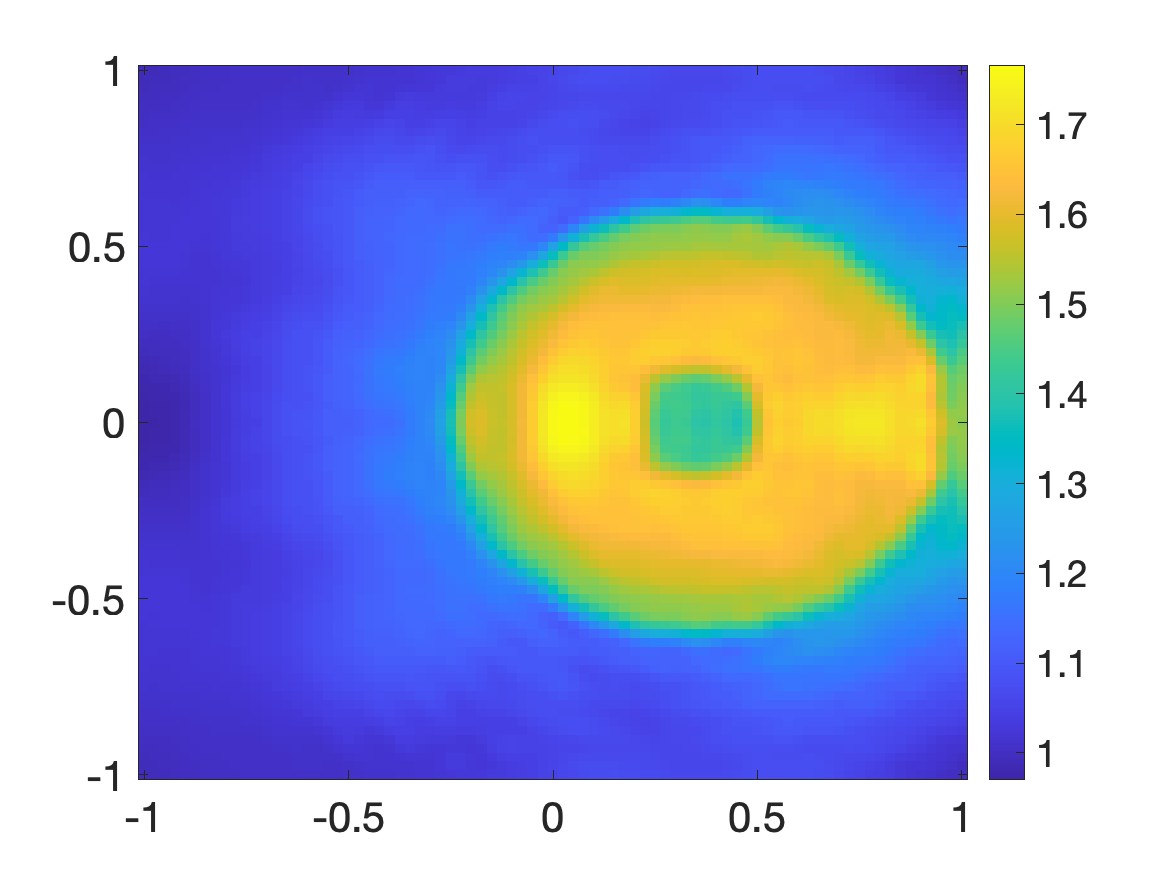}}
	\quad
	\subfloat[\label{fig1c}]{\includegraphics[width=.3\textwidth]{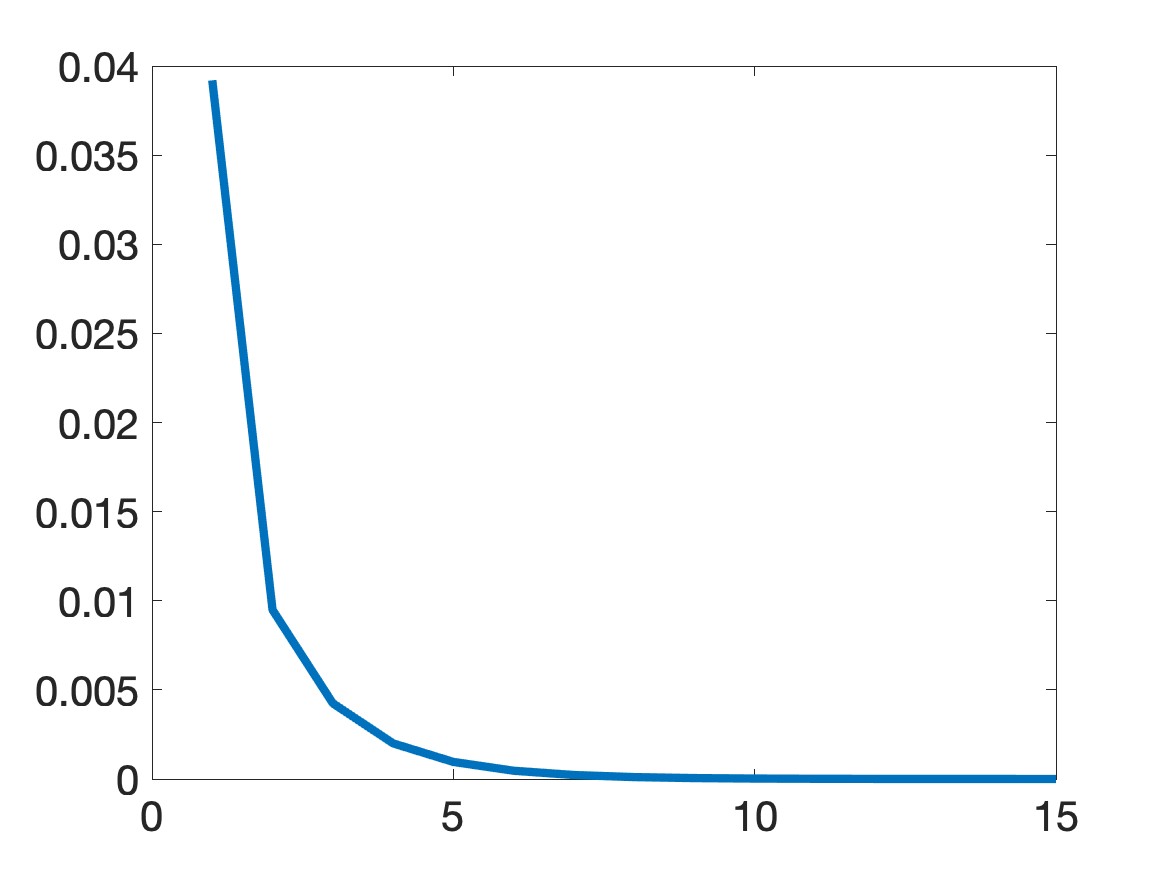}}
	\caption{\label{fig1} (a) and (b) The true and computed initial condition $f^{\rm true}$ and $f^{\rm comp}$ respectively. (c) The consecutive relative difference $\frac{|U_{n + 1} - U_n|_{L^{\infty}(\Omega)}}{\|U_n\|_{L^{\infty}(\Omega)}}$.
	The horizontal axis of this figure is the number of iteration $n$.
	 The boundary data used to reconstruct the function $f$ is corrupted with $\delta = 10\%$ noise.}
\end{figure}

Figure \ref{fig1} illustrates that Algorithm \ref{alg} provides a satisfactory solution to Problem \ref{p}. By comparing the ``donut"-shaped inclusions in both Figure \ref{fig1a} and Figure \ref{fig1b}, we conclude that the computation of the donut's shape and position is quite accurate. Furthermore, the maximum value of the function $f$ within the computed donut is 1.765, which corresponds to a relative error of 11.73\%.
	The approaching zero behavior of the curve in Figure \ref{fig1c} numerically demonstrates the convergence of the Carleman contraction method.

{\bf Test 2.} We consider an intriguing case $f^{\rm true}$ whose graphs features an ``$\Sigma$" shape. The true value of the function $f^{\rm true}(x, y) = 2$ if the point $(x, y)$ in the letter $\Sigma.$  
Otherwise, $f^{\rm true}(x, y) = 1.$
The true and computed source functions $f$ are displayed in Figure \ref{fig2}.

\begin{figure}[h!]
	\subfloat[\label{fig2a}]{\includegraphics[width=.3\textwidth]{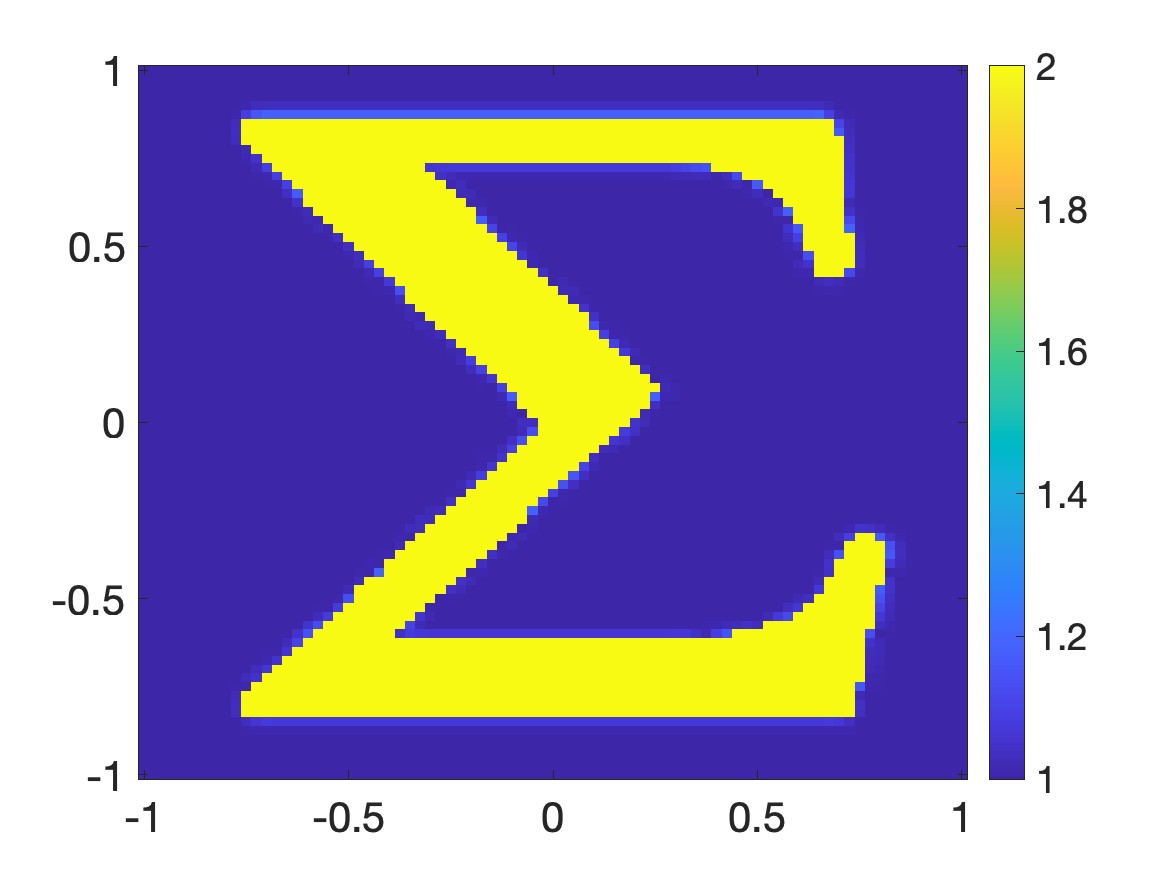}}
	\quad
	\subfloat[\label{fig2b}]{\includegraphics[width=.3\textwidth]{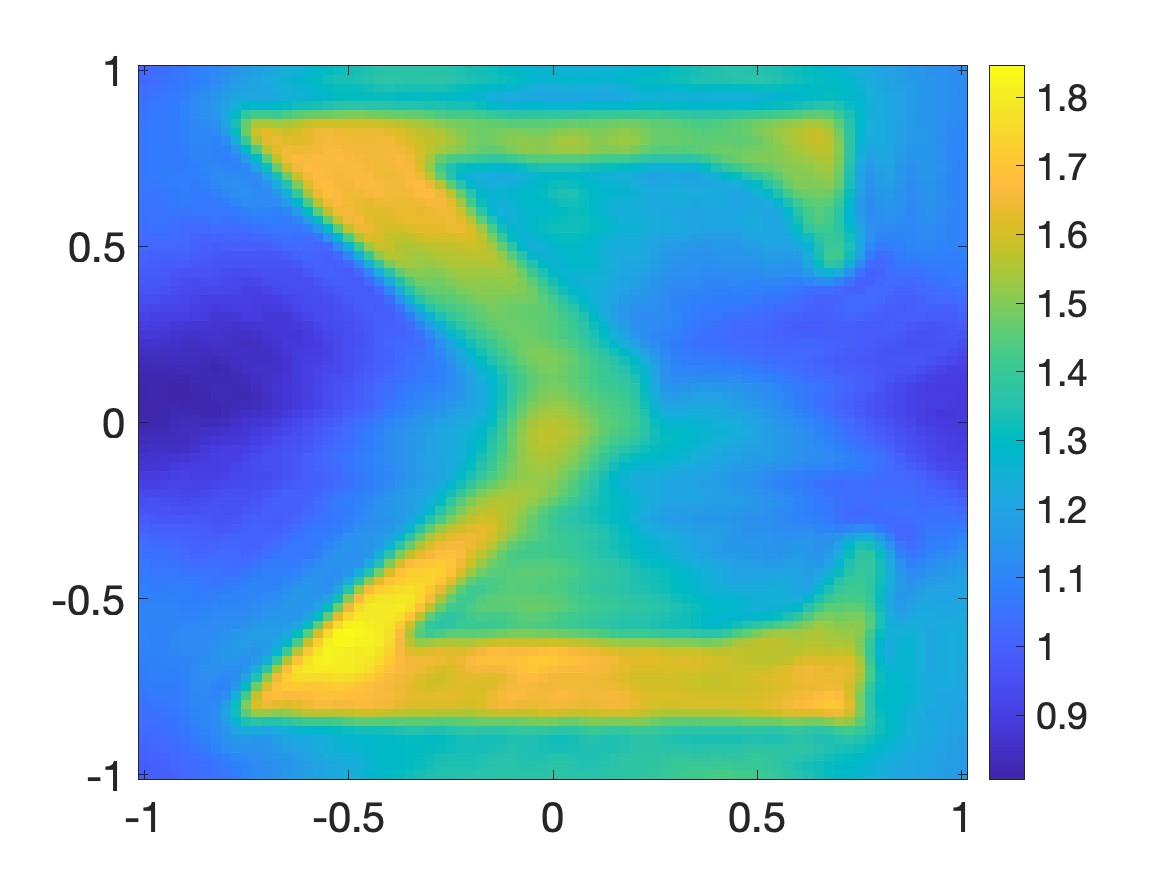}}
	\quad
	\subfloat[\label{fig2c}]{\includegraphics[width=.3\textwidth]{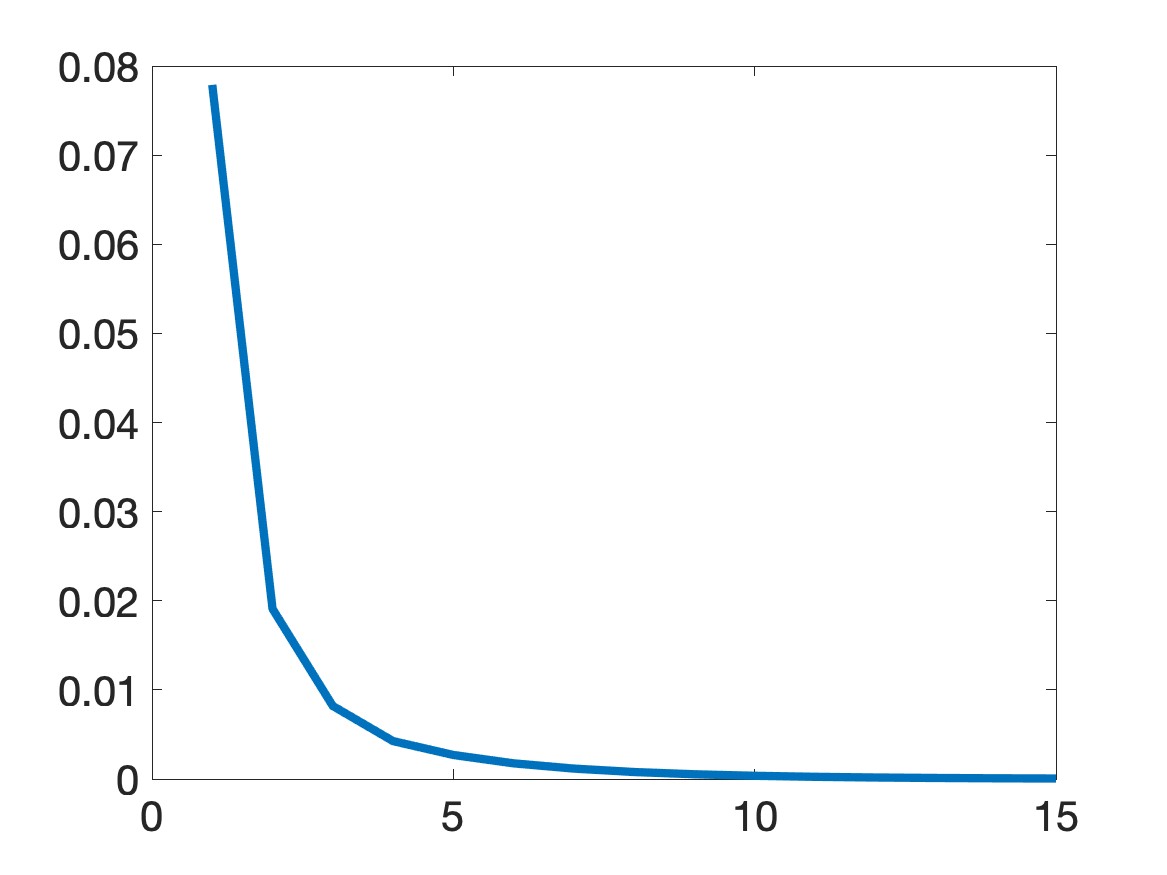}}
	\caption{\label{fig2} (a) and (b) The true and computed initial condition $f^{\rm true}$ and $f^{\rm comp}$ respectively. (c) The consecutive relative difference $\frac{|U_{n + 1} - U_n|_{L^{\infty}(\Omega)}}{\|U_n\|_{L^{\infty}(\Omega)}}$.
	The horizontal axis of this figure is the number of iteration $n$.
	 The boundary data used to reconstruct the function $f$ is corrupted with $\delta = 10\%$ noise.}
\end{figure}

Figure \ref{fig2} presents a numerical solution that is satisfactory for Problem \ref{p}. The letter ``$\Sigma$" and its position are accurately reconstructed. The maximum value of the function $f$ inside the computed "$\Sigma$" is 1.8454, representing a relative error of 7.73\%.
Similar to Test 1, the convergence of the Carleman contraction method is numerically confirmed by Figure \ref{fig2c}.

{\bf Test 3.}
We test a similar experiment to the source function $f$ in Test 2. 
The function $f^{\rm true}$ has a graph looking like the letter ``$\Omega$." The true value of the function $f^{\rm true}(x, y) = 2$ if the point $(x, y)$ in the letter $\Omega.$  
Otherwise, $f^{\rm true}(x, y) = 1.$
The true and computed source functions $f$ are displayed in Figure \ref{fig3}.

\begin{figure}[h!]
	\subfloat[\label{fig3a}]{\includegraphics[width=.3\textwidth]{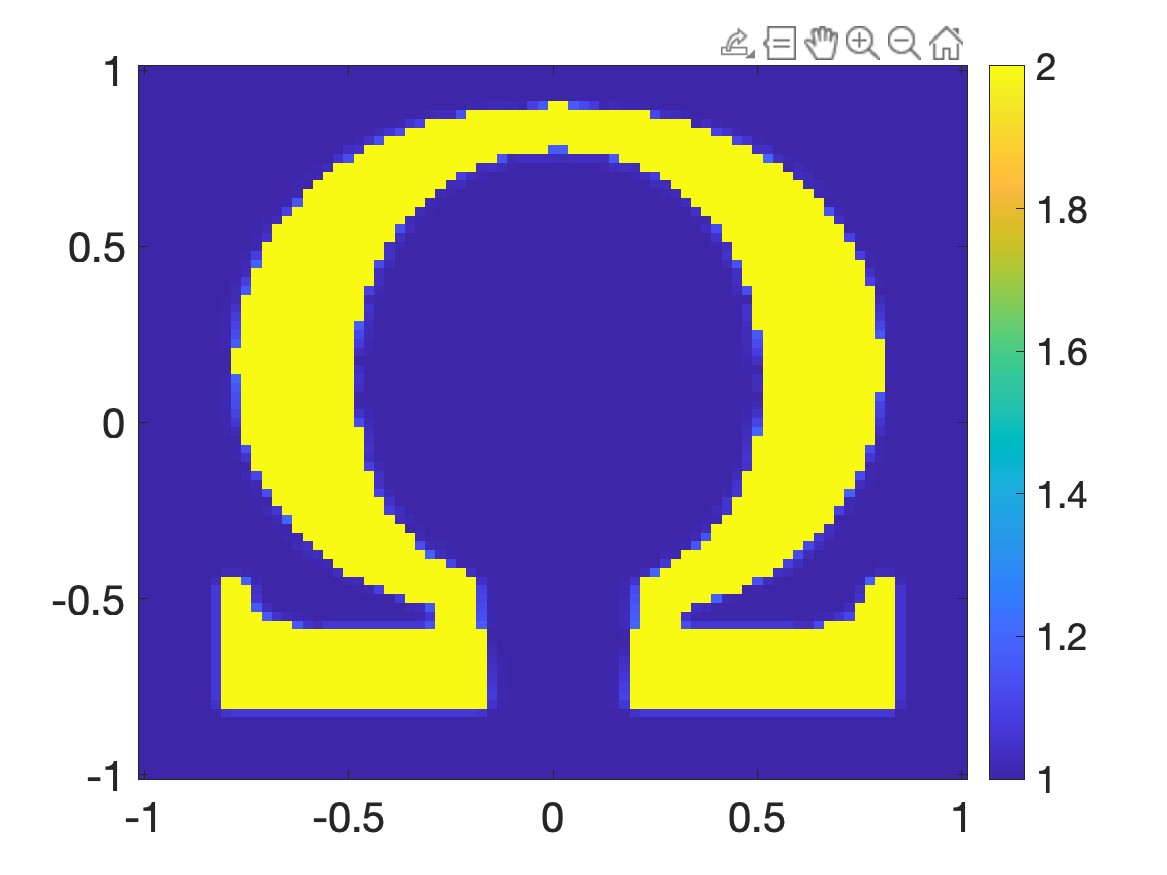}}
	\quad
	\subfloat[\label{fig3b}]{\includegraphics[width=.3\textwidth]{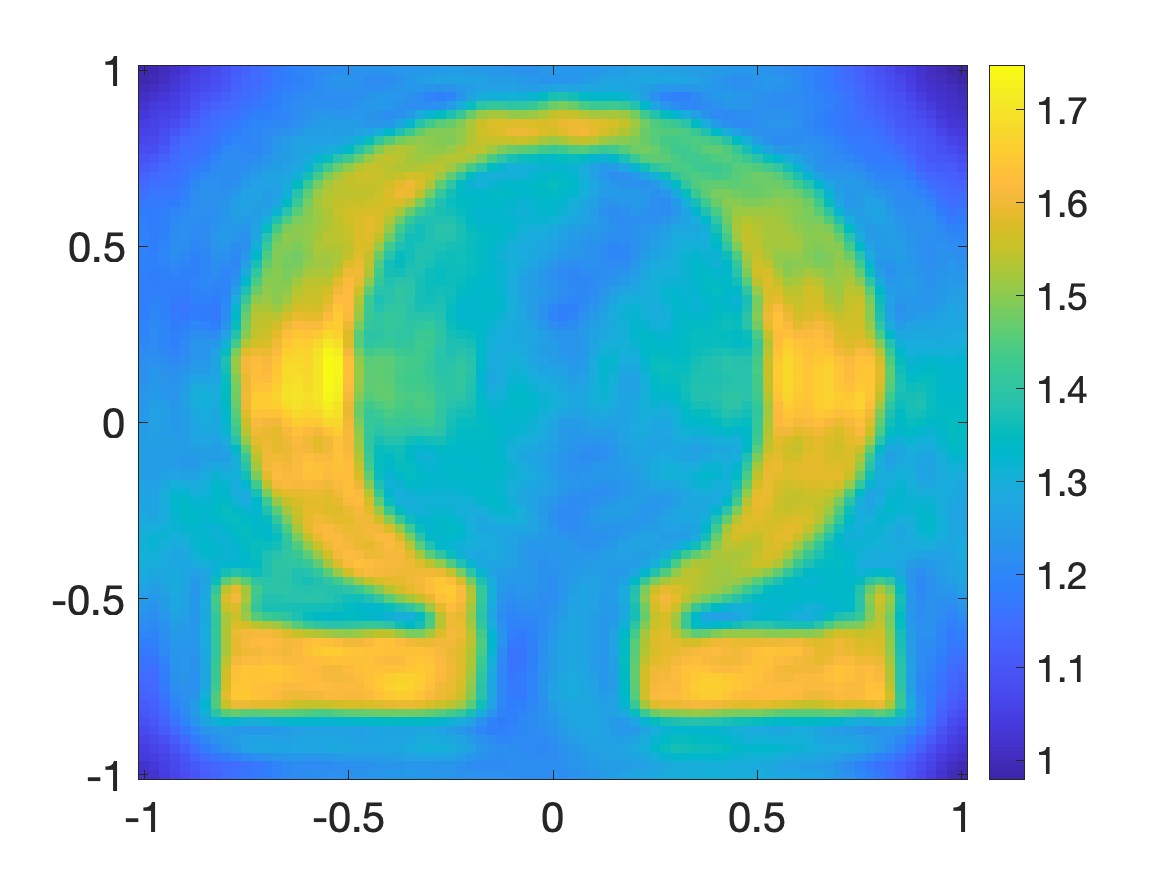}}
	\quad
	\subfloat[\label{fig3c}]{\includegraphics[width=.3\textwidth]{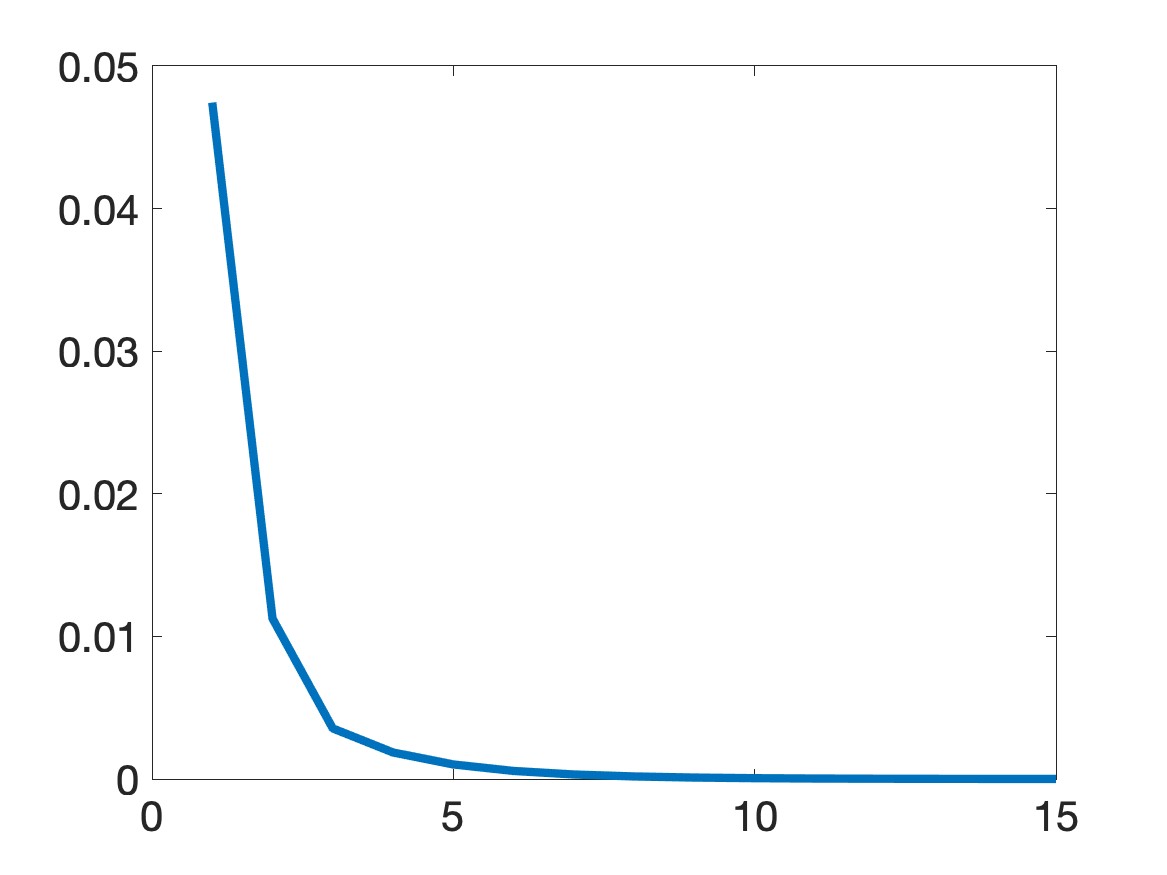}}
	\caption{\label{fig3} (a) and (b) The true and computed initial condition $f^{\rm true}$ and $f^{\rm comp}$ respectively. (c) The consecutive relative difference $\frac{|U_{n + 1} - U_n|_{L^{\infty}(\Omega)}}{\|U_n\|_{L^{\infty}(\Omega)}}$.
	The horizontal axis of this figure is the number of iteration $n$.
	 The boundary data used to reconstruct the function $f$ is corrupted with $\delta = 10\%$ noise.}
\end{figure}

Figure \ref{fig3} presents a numerical solution that is satisfactory for Problem \ref{p}. The letter "$\Omega$" and its position are accurately reconstructed. The maximum value of the function $f$ inside the computed "$\Omega$" is 1.7473, representing a relative error of 12.64\%.
The convergence of the Carleman contraction method is numerically confirmed by in Figure \ref{fig3c}.

{\bf Test 4.} In this test, we consider a more complicated circumstance in which the graph of the true function $f$ has two ``inclusions".
 Each inclusion looks like a horizontal line segment. 
 The value of of function $f^{\rm true}$ is 4 in the line on the top and $3$ inside the line in the bottom.  
 The true and computed source functions $f$ are displayed in Figure \ref{fig4}.
 
 \begin{figure}[h!]
	\subfloat[\label{fig4a}]{\includegraphics[width=.3\textwidth]{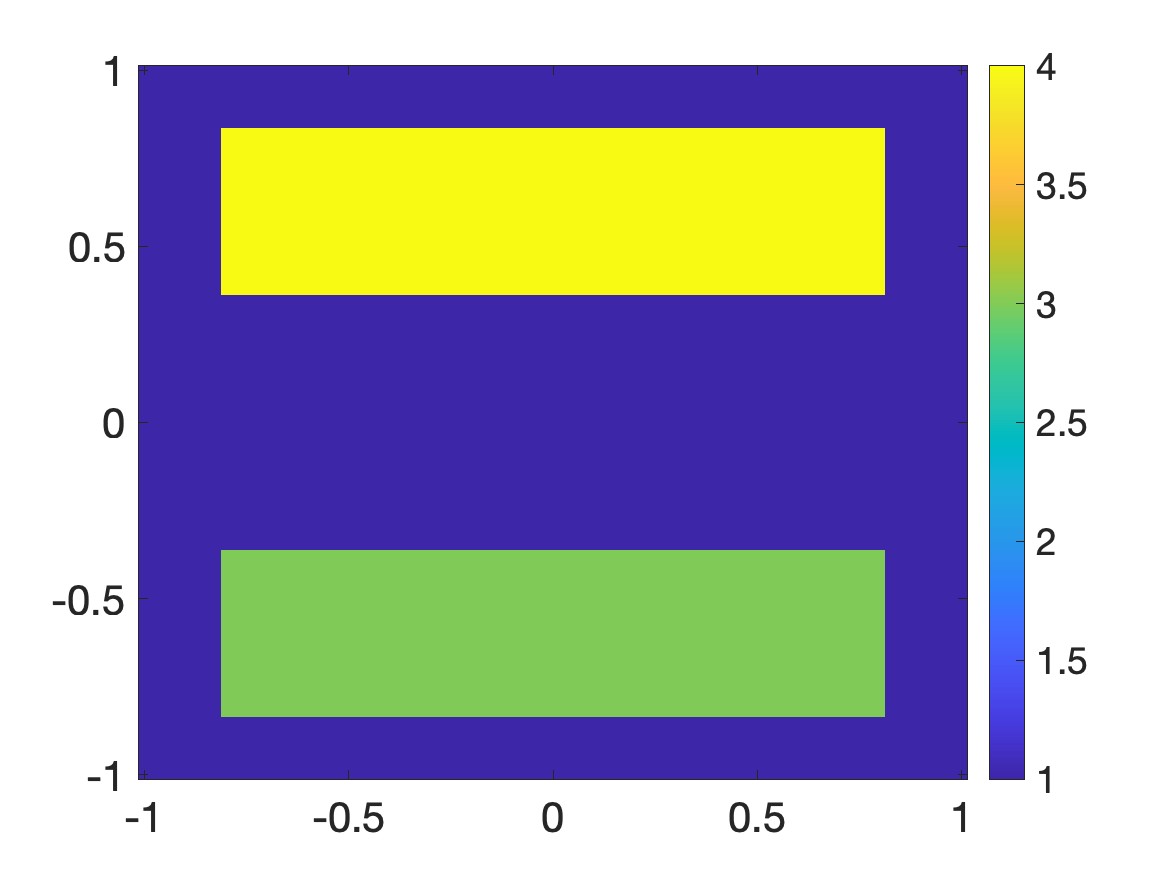}}
	\quad
	\subfloat[\label{fig4b}]{\includegraphics[width=.3\textwidth]{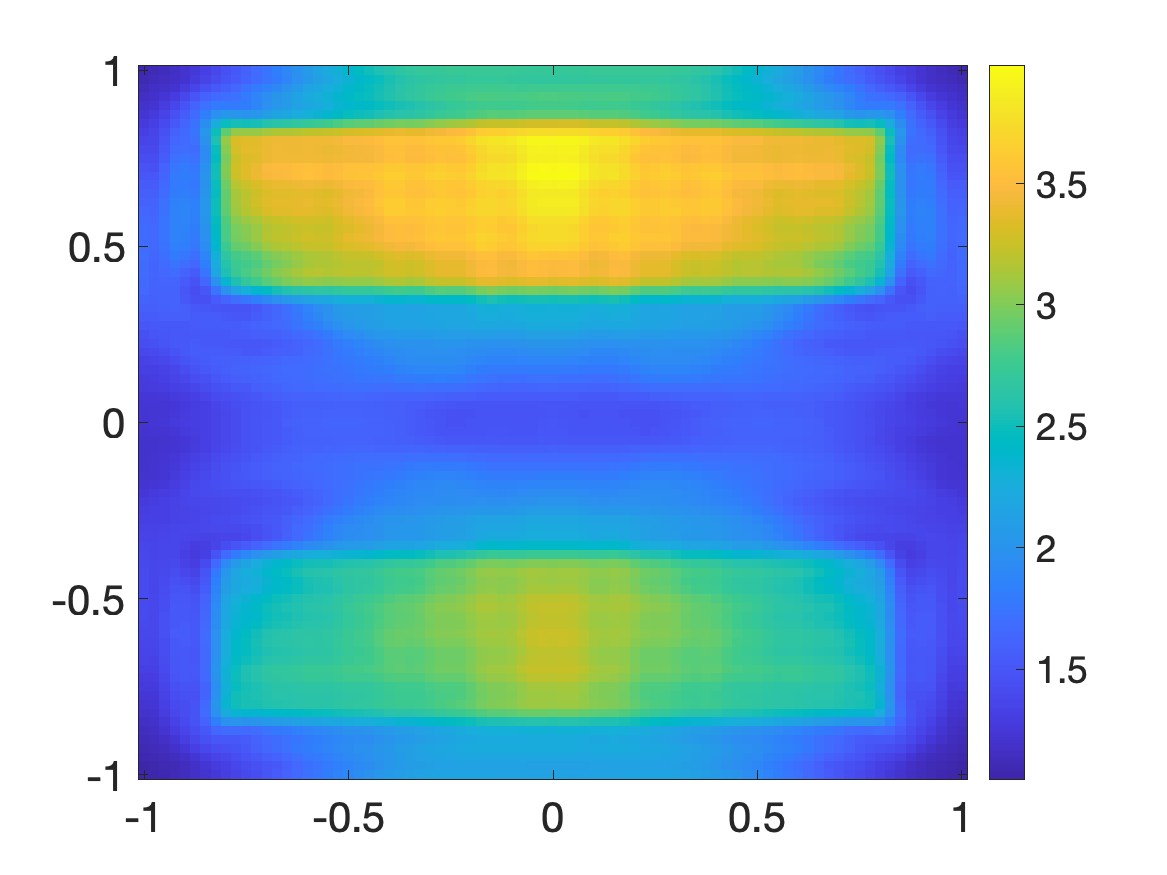}}
	\quad
	\subfloat[\label{fig4c}]{\includegraphics[width=.3\textwidth]{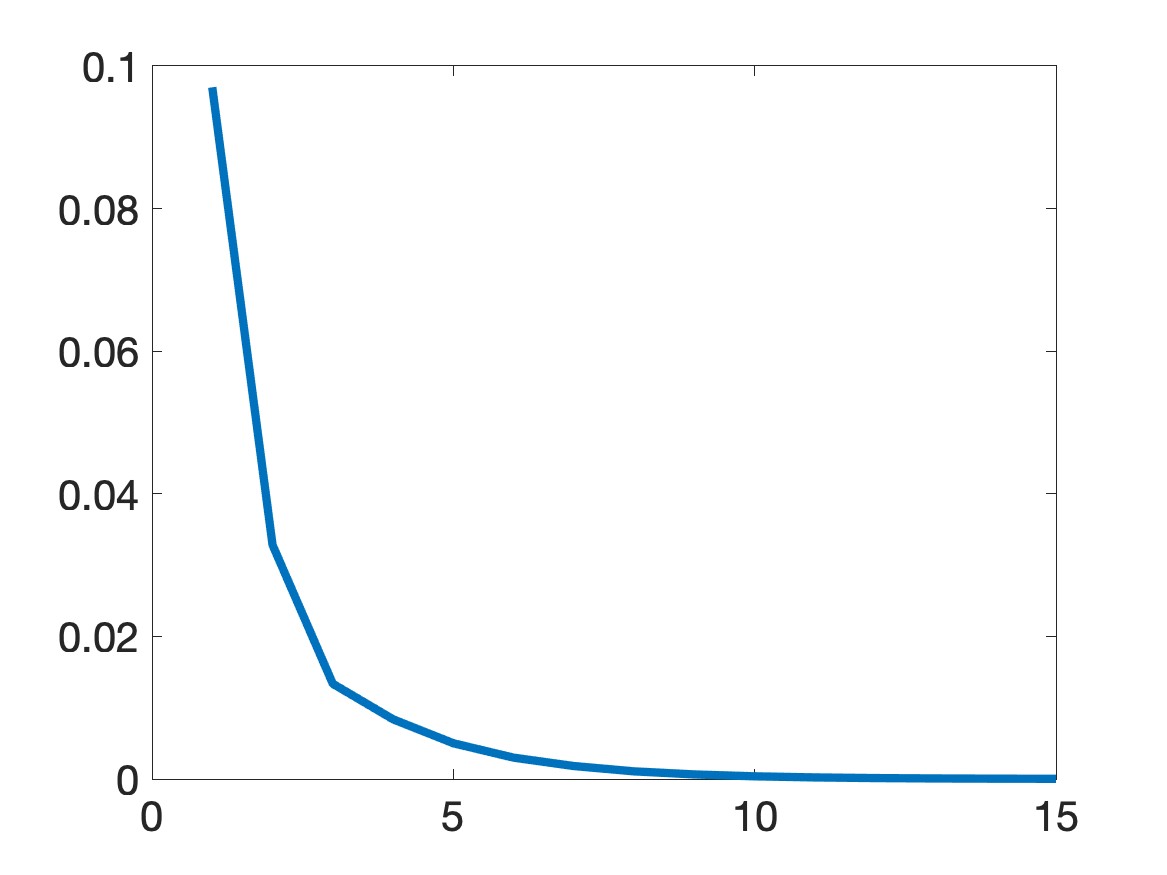}}
	\caption{\label{fig4} (a) and (b) The true and computed initial condition $f^{\rm true}$ and $f^{\rm comp}$ respectively. (c) The consecutive relative difference $\frac{|U_{n + 1} - U_n|_{L^{\infty}(\Omega)}}{\|U_n\|_{L^{\infty}(\Omega)}}$.
	The horizontal axis of this figure is the number of iteration $n$.
	 The boundary data used to reconstruct the function $f$ is corrupted with $\delta = 10\%$ noise.}
\end{figure}
  
Our algorithm works well for this test.
It is evident that both ``horizontal inclusions" are successfully identified. The peak value of the computed function $f$ inside the ``top" inclusion is 3.9845 (relative error 0.39\%).
The peak value of the computed function $f$ inside the ``bottom" inclusion is 3.25782 (relative error 8.59\%).

\begin{Remark}
	To emphasize that our method is independent of the unknown damping coefficient $a$, we use different $a$ in the tests above.
	\begin{enumerate}
		\item 
		The unknown function $a^{\rm true}$ in Test 1 is given by 
\[
	a^{\rm true}(x, y) = \left\{
	\begin{array}{ll}
		2 &\mbox{if } 0.15^2 < (x - 0.35)^2 + y^2 < 0.6^2\\
		1 &\mbox{otherwise},
	\end{array}
	\right.
	\quad
	\mbox{for all } (x, y) \in \Omega.
\]
\item The unknown function $a^{\rm true}$ in Test 2 is given by
\[
	a^{\rm true}(x, y) = |y^2 - x|
	\quad
	\mbox{for all } (x, y) \in \Omega.
\]
\item The unknown function $a^{\rm true}$ in Test 3 is given by
\[
	a^{\rm true}(x, y) = x^2
	\quad
	\mbox{for all } (x, y) \in \Omega.
\]
\item The unknown function $a^{\rm true}$ in Test 4 is given by
\[
	a^{\rm true}(x, y) = \left\{
	\begin{array}{ll}
		2e^{\frac{r^2}{r^2 - 1}} &\mbox{if } r < 1\\
		1 &\mbox{otherwise},
	\end{array}
	\right.
	\quad
	\mbox{for all } (x, y) \in \Omega.	
\]
where $r = r(x, y) = \sqrt{\frac{x^2}{0.5^2} + \frac{y^2}{0.25^2}}$.
	\end{enumerate}		
\end{Remark}
	
	Upon calculating the vector value $U^{\rm comp} = (u_1^{\rm comp}, \dots, u_N^{\rm comp})$ at Step \ref{s11}, one might intuitionally consider the problem of reconstructing the unknown coefficient $a(x, y)$, $(x, y) \in \Omega$, by combining equations \eqref{2,1} and \eqref{2.1}. More precisely, one could suggest the following formula
\begin{equation}
a(x, y) = -\frac{\big[\sum_{l = 1}^N u_l^{\rm comp}(x, y) \Psi_l'(0)\big]\big[\sum_{l = 1}^N u_l^{\rm comp}(x, y) \Psi_l(0)\big]}{\big|\sum_{l = 1}^N u_l^{\rm comp}(x, y) \Psi_l(0)\big|^2 +\eta^2}
\quad
\mbox{for all } (x, y) \in \Omega.
\label{4.6}
\end{equation}
However, we find that equation \eqref{4.6} is not universally successful. Our numerical observations suggest that its applicability is heavily contingent on the nature of the function $f^{\rm true}$. If $f^{\rm true}$ is smooth, equation \eqref{4.6} can reliably reconstruct the coefficient $a$. Conversely, if $f^{\rm true}$ lacks continuity, the formula presented in \eqref{4.6} fails.
In the next two tests, we show the reconstruction of both $f$ and $a$ when $f$ is smooth.

{\bf Test 5.}
In this test, we set 
\[
	f^{\rm true}(x, y) = y^2 - x + 5 
\]
and
\[
	a^{\rm true}(x, y) =
	\left\{
		\begin{array}{ll}
			0 &\mbox{if } \max\{\sqrt{2}|x - 0.4|, |y|\} > 0.5 \,\mbox{or } \max\{|x - 0.4|,|y|\}_1 < 0.12,\\
			2 &\mbox{otherwise,}
		\end{array}
	\right.
\]
for all $(x, y) \in \Omega$.
The true and reconstructed of these functions are displayed in Figure \ref{fig5}.

\begin{figure}[h!]
	\centering
	\subfloat[]{\includegraphics[width = .35\textwidth]{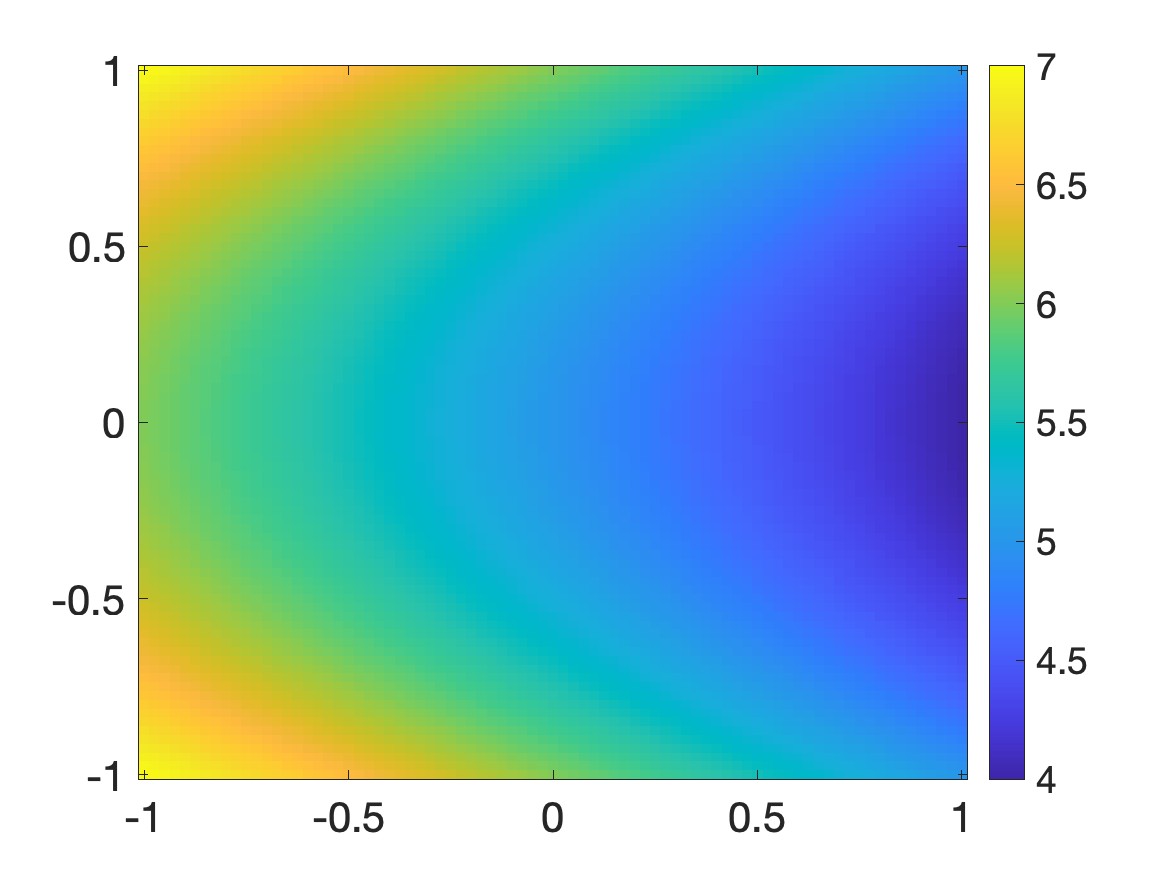}}
	\quad
	\subfloat[]{\includegraphics[width = .35\textwidth]{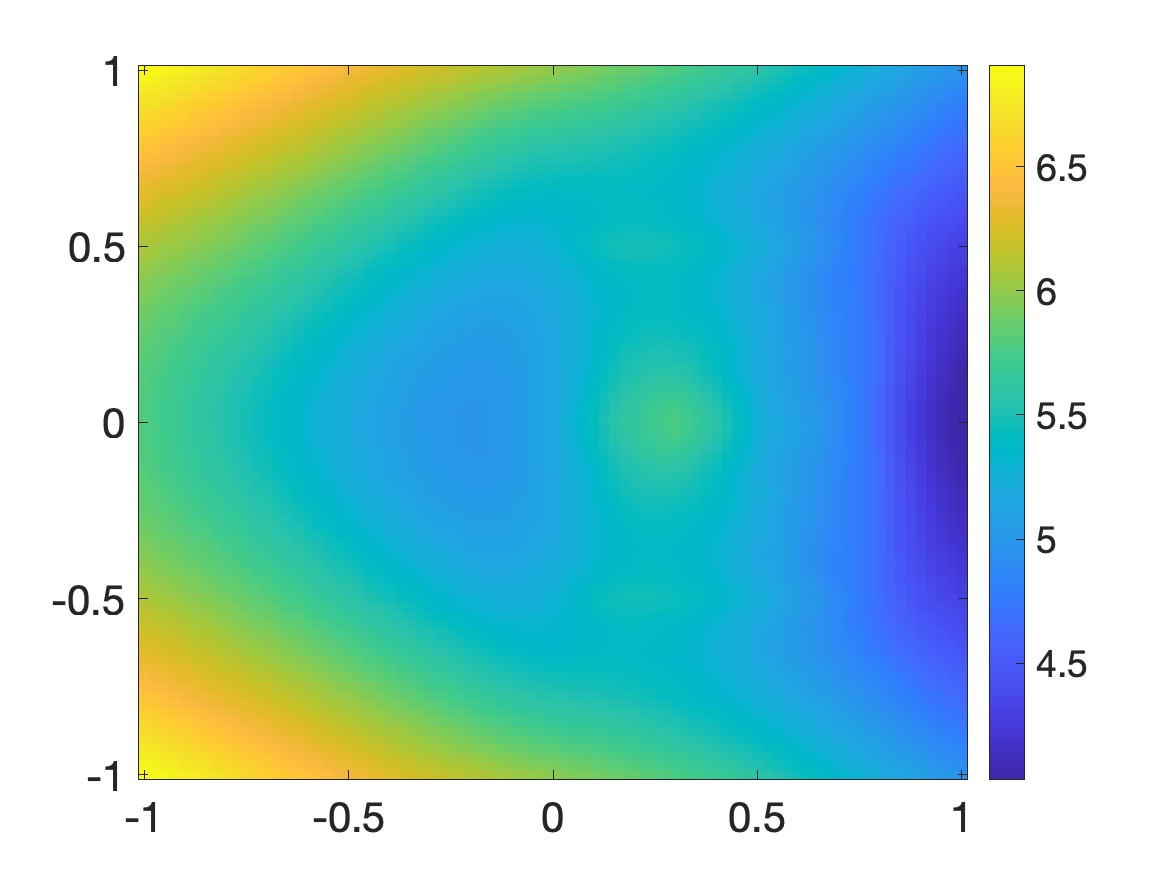}}
	
	\subfloat[]{\includegraphics[width = .35\textwidth]{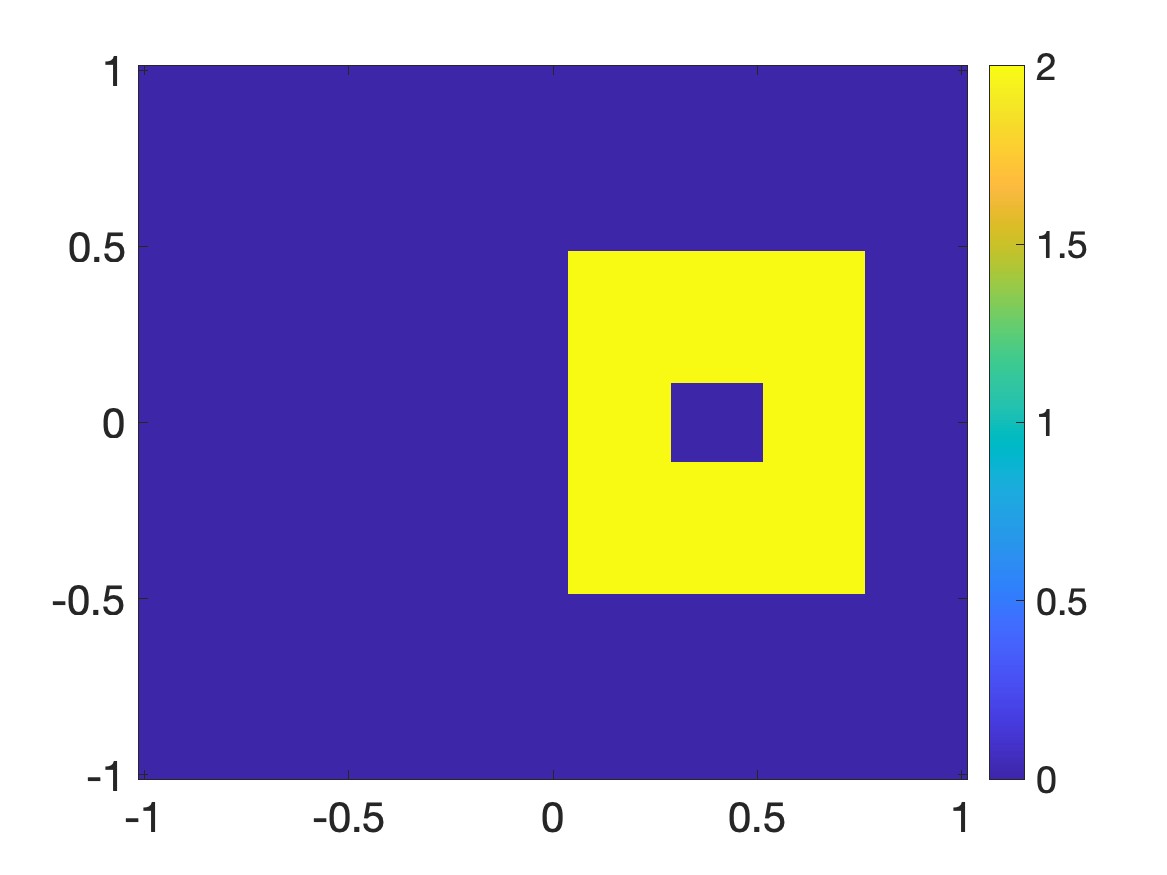}}
	\quad
	\subfloat[]{\includegraphics[width = .35\textwidth]{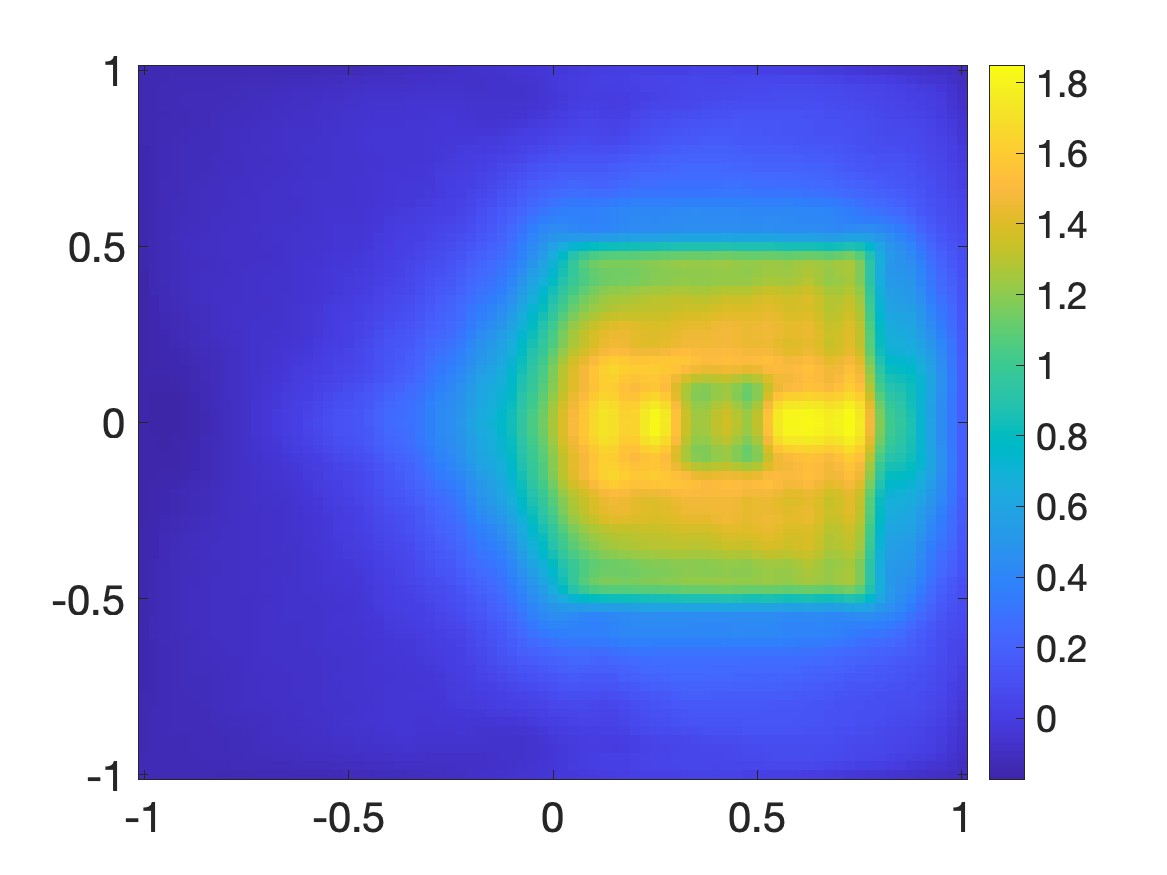}}
	\caption{\label{fig5}
		(a) and (b) the true and computed initial conditions $f$.
		(c) and (d) the true and computed damping coefficients $a$.
		The data incorporated in this test includes a noise level of 10\%.
	}	
\end{figure}

The relative error for the initial condition, calculated as $\frac{\|f^{\rm comp} - f^{\rm true}\|_{L^2(\Omega)}}{\|f^{\rm true}\|_{L^2(\Omega)}}$, is 5.47\%, which falls below the noise level. Although the $L^2$ relative error in calculating the damping coefficient $a$ is large, the reconstructed maximum value of $a$ within the inclusion resembling a square with a void is still considered acceptable. Inside the inclusion, the computed maximum value of $a^{\rm comp}$ is 1.8474, corresponding to a relative error of 7.63\%.

{\bf Test 5.}
In this test, we set 
\[
	f^{\rm true}(x, y) = x - y + 7 
\]
and
\[
	a^{\rm true}(x, y) =
	\left\{
		\begin{array}{ll}
			2 &\mbox{if } (x - 0.55)^2 + (y - 0.55)^2 < 0.4^2,\\
			4 &\mbox{if }  (x + 0.55)^2 + (y - 0.55)^2 < 0.4^2,\\
			3 &\mbox{if } (x + 0.55)^2 + (y + 0.55)^2 < 0.4^2,\\
			1 &\mbox{otherwise}
		\end{array}
	\right.
\]
for all $(x, y) \in \Omega$.
This test is interesting since we are solving a nonlinear problem when the values of both unknown $a$ and $f$ are high. The true and reconstructed of these functions are displayed in Figure \ref{fig6}.

\begin{figure}[h!]
	\centering
	\subfloat[]{\includegraphics[width = .35\textwidth]{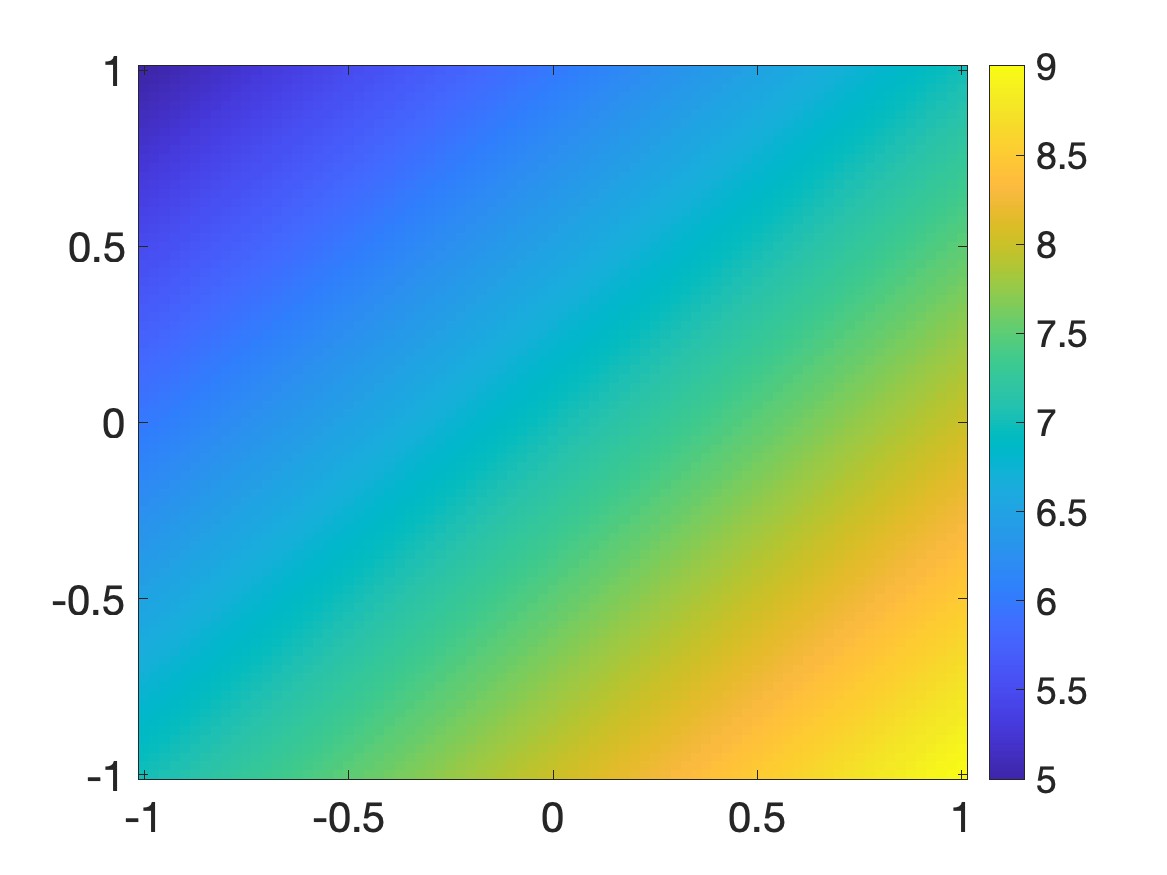}}
	\quad
	\subfloat[]{\includegraphics[width = .35\textwidth]{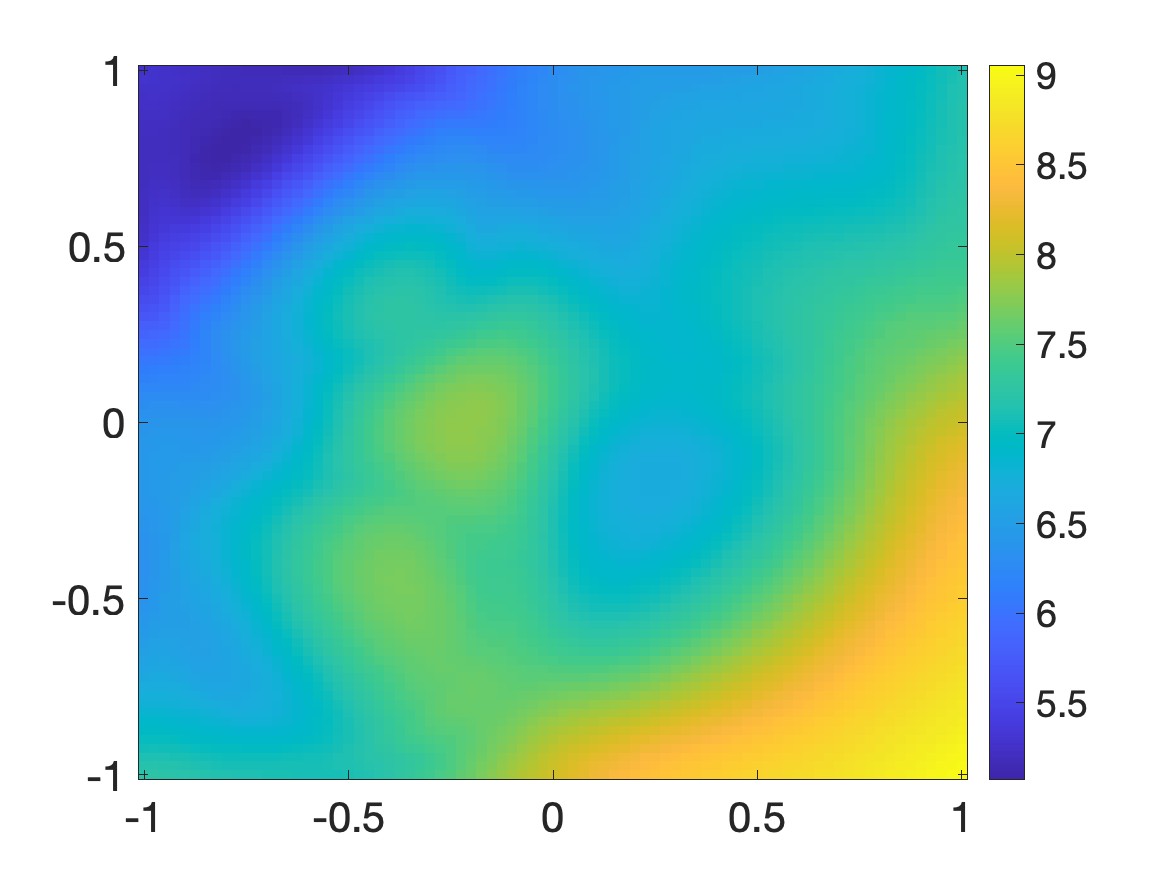}}
	
	\subfloat[]{\includegraphics[width = .35\textwidth]{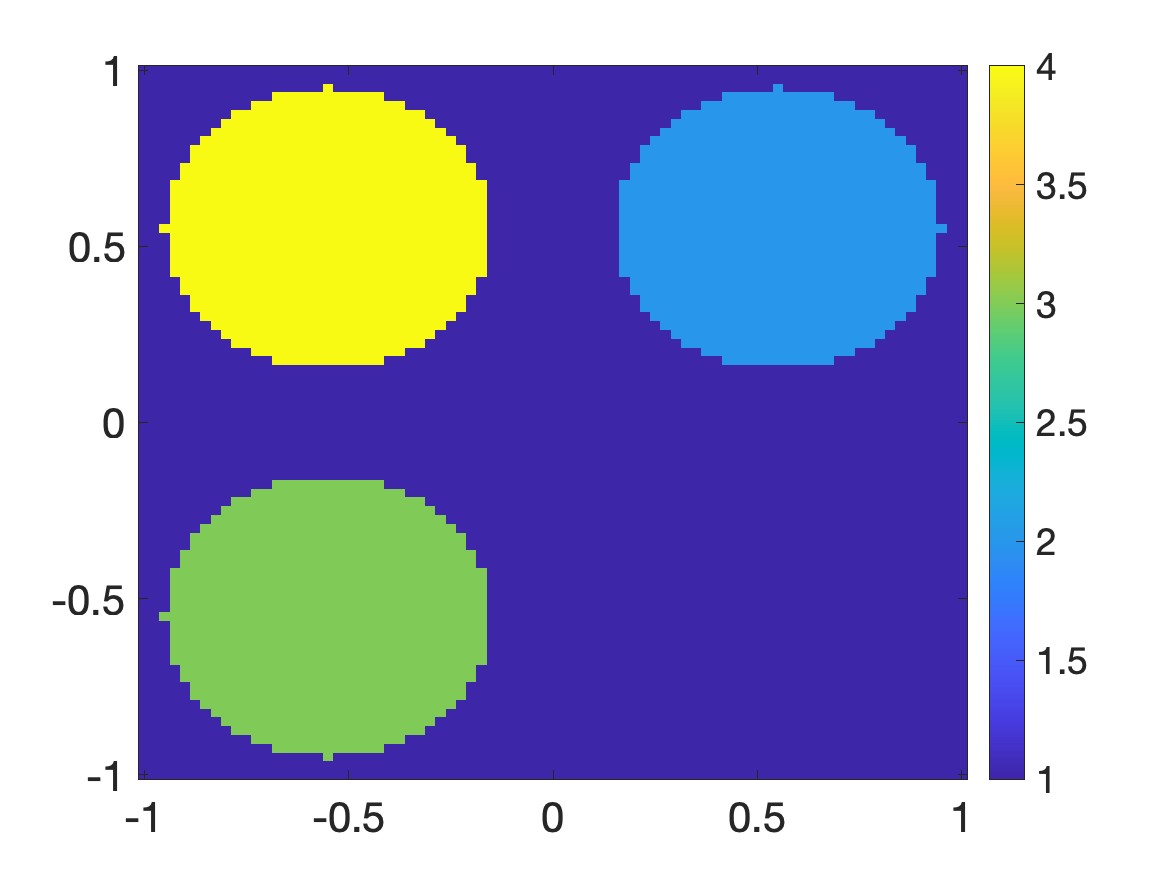}}
	\quad
	\subfloat[]{\includegraphics[width = .35\textwidth]{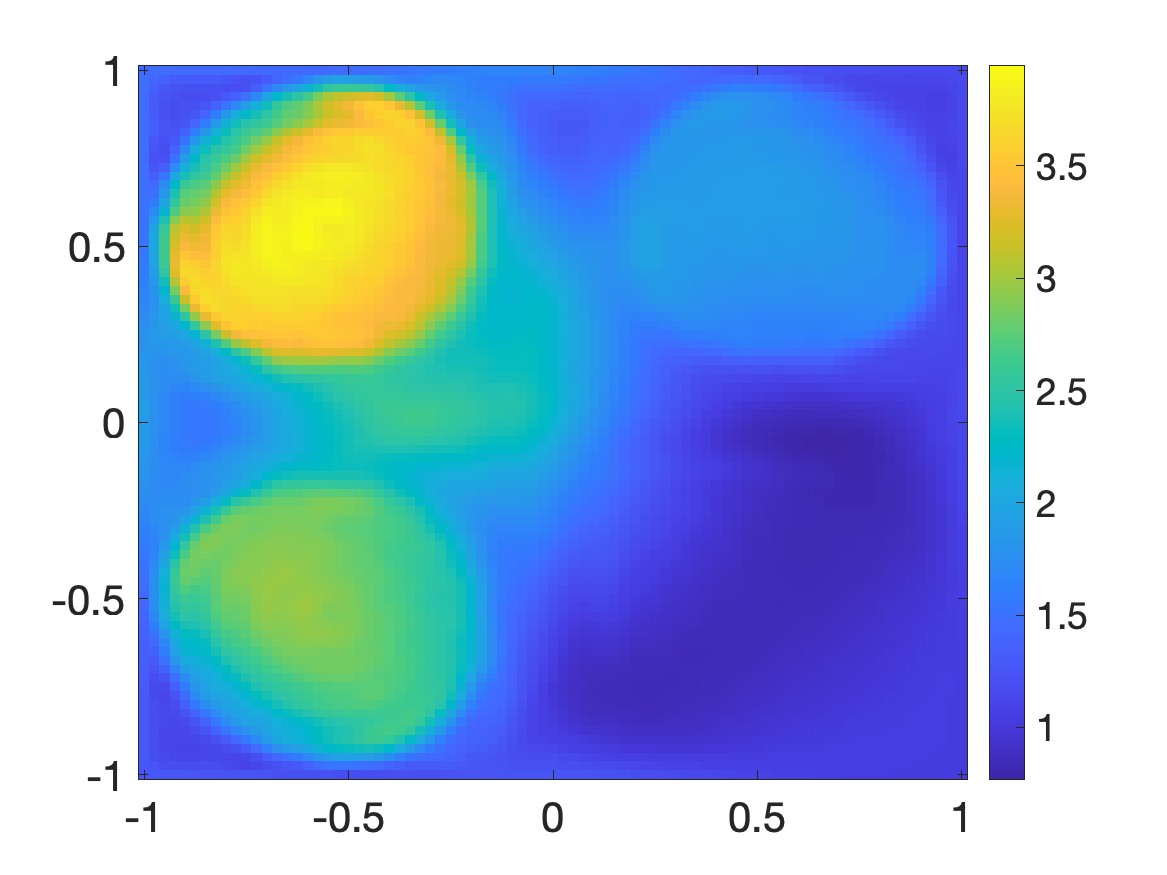}}
	\caption{\label{fig6}
		(a) and (b) the true and computed initial conditions $f$.
		(c) and (d) the true and computed damping coefficients $a$.
		The data incorporated in this test includes a noise level of 10\%.
	}	
\end{figure}

The relative error for the initial condition, calculated as $\frac{\|f^{\rm comp} - f^{\rm true}\|{L^2(\Omega)}}{\|f^{\rm true}\|{L^2(\Omega)}}$, is 5.37\%, which falls below the noise level. As in Test 5, the $L^2$ relative error in calculating the damping coefficient $a$ is large. However, the reconstructed maximum value of $a$ within each inclusion is still satisfactory. Inside the top left inclusion, the computed maximum value of $a^{\rm comp}$ is 3.9336 (relative error 1.66\%). Inside the top right inclusion, the computed maximum value of $a^{\rm comp}$ is 1.89312 (relative error 5.34\%). Inside the bottom left inclusion, the computed maximum value of $a^{\rm comp}$ is 2.9731 (relative error 0.9\%). 


 \section*{Acknowledgement}
The works of TTL and LHN were partially supported  by National Science Foundation grant DMS-2208159, and  by funds provided by the Faculty Research Grant program at UNC Charlotte Fund No. 111272,  and by
the CLAS small grant provided by the College of Liberal Arts \& Sciences, UNC Charlotte. 
The works of HP were supported in part by National Science Foundation grant DMS-2150179. LVN was partially supported by the National Science Foundation grant DMS-1616904.


\end{document}